\documentclass[10pt]{article}
\usepackage{tabularx}
\usepackage{listings}
\usepackage{stmaryrd}
\usepackage{graphicx}
\usepackage{cmap}
\usepackage{amssymb,amsmath,amsthm}
\usepackage{euscript}
\usepackage[normalem]{ulem}
\usepackage{comment}
\usepackage{verbatim}
\usepackage[dvipsnames]{xcolor}
\usepackage{ragged2e}
\usepackage{enumitem}
\usepackage{mathrsfs}
\usepackage{stmaryrd}
\usepackage{tikz-cd}
\usepackage[hidelinks,colorlinks=true,linkcolor=red,citecolor=blue]{hyperref}
\usetikzlibrary{babel}

\newcommand{\R}{\mathbb{R}}
\newcommand{\Z}{\mathbb{Z}}

\newcommand{\C}{\mathbb{C}}
\newcommand{\bO}{\mathbb{O}}
\newcommand{\bG}{\mathbb{G}}

\renewcommand{\P}{\mathbb{P}}
\newcommand{\bA}{\mathbb{A}}

\newcommand{\e}{\varepsilon}
\renewcommand{\phi}{\varphi}

\newcommand{\s}{\sigma}
\newcommand{\La}{\Lambda}
\newcommand{\om}{\omega}
\newcommand{\ga}{\gamma}
\newcommand{\la}{\lambda}
\newcommand{\p}{\partial}

\newcommand{\De}{\Delta}

\newcommand{\Om}{\Omega}

\renewcommand{\O}{\mathcal{O}}

\newcommand{\B}{\mathcal{B}}

\newcommand{\cS}{{\mathcal{S}}}

\renewcommand{\Im}{\text{Im}}

\newcommand{\Ann}{\text{Ann}}

\newcommand{\Id}{\text{Id}}

\newcommand{\Span}{\text{Span}}

\newcommand{\PSL}{\text{PSL}}
\newcommand{\Sp}{\text{Sp}}

\newcommand{\mr}{\leqslant}
\newcommand{\br}{\geqslant}
\newcommand{\f}{\frac}

\renewcommand{\l}{\limits}
\renewcommand{\o}{\overline}
\renewcommand{\sb}{\subseteq}
\newcommand{\sbn}{\subsetneq}

\newcommand{\ld}{\ldots}

\newcommand{\w}{\wedge}
\renewcommand{\t}{\widetilde}

\newcommand{\vn}{\varnothing}
\newcommand{\inj}{\hookrightarrow}
\newcommand{\sur}{\twoheadrightarrow}
\newcommand{\ot}{\otimes}
\newcommand{\op}{\oplus}

\renewcommand{\ld}{\ldots}
\renewcommand{\o}{\overline}
\newcommand{\cd}{\cdot}

\newcommand{\h}{\widehat}
\newcommand{\bs}{\backslash}
\newcommand{\circar}{\circlearrowright}

\newcommand{\<}{\langle}
\renewcommand{\>}{\rangle}

\renewcommand{\char}{\text{char}\,}

\newcommand{\IGr}{\text{IGr}}

\newcommand{\Lie}{\text{Lie}\,}
\newcommand{\Stab}{\text{Stab}}
\newcommand{\SL}{\text{SL}}
\newcommand{\GL}{\text{GL}}

\newcommand{\m}{\mathfrak{m}} 

\renewcommand{\sl}{\mathfrak{sl}}

\newcommand{\gl}{\mathfrak{gl}}
\newcommand{\g}{\mathfrak{g}}
\newcommand{\q}{\mathfrak{q}}

\renewcommand{\u}{\mathfrak{u}}
\renewcommand{\h}{\mathfrak{h}}

\newcommand{\fb}{\mathfrak{b}}
\newcommand{\ad}{\text{ad}}
\newcommand{\Ad}{\text{Ad}}
\newcommand{\n}{\mathfrak{n}}
\newcommand{\ft}{\mathfrak{t}}

\renewcommand{\sp}{\mathfrak{sp}}

\newcommand{\fs}{\mathfrak{s}}

\newcommand{\se}{\mathsf{e}}

\newcommand{\scrX}{{\mathscr X}}

\renewcommand{\a}{\eta}
\renewcommand{\b}{\theta}
\newcommand{\ba}{\boldsymbol{\alpha}}
\newcommand{\al}{\alpha}
\newcommand{\be}{\beta}

\newcommand{\iso}{\xrightarrow{
    \,\smash{\raisebox{-0.65ex}{\ensuremath{\scriptstyle\sim}}}\,}}

\definecolor{Red}{rgb}{0.7, 0,0}

\theoremstyle{definition}
\newtheorem{lem}{Lemma}[subsection]
\newtheorem{theo}[lem]{Theorem}
\newtheorem{prop}[lem]{Proposition}
\newtheorem{cor}[lem]{Corollary}
\newtheorem{ex}[lem]{Example}
\newtheorem{rem}[lem]{Remark}
\newtheorem{defi}[lem]{Definition}

\newtheorem{fact}[lem]{Fact}

\title 
{Lagrangian Subvarieties of Hyperspherical Varieties Related to $G_2$}
\author{Nikolay Kononenko}

\begin{document}
	
\maketitle

\begin{abstract}
  We consider two $S$-dual hyperspherical varieties of the group $G_2\times\SL(2)$:
  an equivariant slice for $G_2$, and the symplectic representation of $G_2 \times \SL_2$
  in the odd part of the basic classical Lie superalgebra $\g(3)$. For these varieties we
  check the equality of numbers of irreducible components of their Lagrangian subvarieties (zero
  levels of the moment maps of Borel subgroups' actions) conjectured in~\cite{fgt}. 
\end{abstract}



\tableofcontents


\section{Introduction} \label{intr}

\subsection{Setup} \label{intr form}

We consider a nilpotent element $e$ in the exceptional simple complex Lie algebra $\g_2$,
corresponding to a short root vector, so that $e$ lies in the 8-dimensional nilpotent orbit
${\mathbb O}_8\subset\g_2$. We fix an $\sl_2$-triple $(e,h,f)$ in $\g_2$. It defines a
Slodowy slice $\cS_e$ to ${\mathbb O}_8$. The centralizer of this $\sl_2$-triple in the
ambient group $G_2$ is a subgroup $\SL(2)\subset G_2$. We obtain a natural
$G_2\times\SL(2)$-action on the {\em equivariant slice} $\scrX=G_2\times\cS_e$. The equivariant
slice carries a $G_2\times\SL(2)$-invariant symplectic structure, see e.g.~\cite{l1,l2}
or~\cite[\S3.4]{bzsv}.
This slice is equivariantly symplectomorphic to a twisted cotangent bundle $T^*_\psi(G_2/U)$
for a 4-dimensional unipotent subgroup of $G_2$.
Moreover, the Hamiltonian action of $G_2\times\SL(2)$ on $G_2\times\cS_e$ is {\em coisotropic},
i.e.\ a typical $G_2\times\SL(2)$-orbit is isotropic. It also satisfies the
conditions~\cite[\S3.5.1]{bzsv} of {\em hypersphericity}. But it has a nontrivial anomaly (cf.~\cite[\S4.1]{bdfrt}), so its
$S$-dual variety $\scrX^\vee$ (cf.~\cite[\S4]{bzsv}) carries a hamiltonian action not of the
Langlands dual group $(G_2\times\SL(2))^\vee$, but rather of the {\em metaplectic} dual group that
is again isomorphic to $G_2\times\SL(2)$. It turns out that $\scrX^\vee$ is just a linear
symplectic representation of $G_2\times\SL(2)$, namely $\scrX^\vee=\C^7\otimes\C^2$, where
$\C^7$ is the 7-dimensional fundamental representation of $G_2$, while $\C^2$ is the 2-dimensional
fundamental representation of $\SL(2)$. Note that $G_2\times\SL(2)$ is the even part of the
exceptional basic classical Lie supergroups $G(3)$, and $\C^7\otimes\C^2$ is the odd part of
its Lie superalgebra $\g(3)$.

Set $\t{G}=G_2\times\SL(2)$ and choose a Borel subgroup $\t{B}\subset \t{G}$. Let $\t{\fb}$ denote the Lie algebra of $\t{B}$, and let $\mu_{\scrX}\colon\scrX\to {\t{\fb}}^\ast$ (resp.\ 
$\mu_{\scrX^\vee}\colon\scrX^\vee\to {\t{\fb}}^\ast$) denote the corresponding moment map.
Let $\Lambda_{\scrX}\subset\scrX$ (resp.\ $\Lambda_{\scrX^\vee}\subset\scrX^\vee$) denote the
zero level of the moment map $\mu_{\scrX}$ (resp.\ $\mu_{\scrX^\vee}$). 
Then according to~\cite[Conjecture 1.1.1]{fgt}, both $\Lambda_{\scrX}\subset\scrX$ and
$\Lambda_{\scrX^\vee}\subset\scrX^\vee$ should be Lagrangian subvarieties, and they should have the
same number of irreducible components. In fact, this conjecture is checked in~\cite{fgt}
for all basic classical supergroups (that is for $\t{G}^\vee$ the even part of such supergroup
and $\scrX^\vee$ the odd part of its Lie superalgebra) {\em except} for $G(3)$.

We settle the remaining case in the present note.

\begin{theo} \label{theo}
  The varieties $\La_\scrX$ and $\La_{\scrX^\vee}$ are both Lagrangian and each has $7$
  irreducible components.
\end{theo}

\subsection{Organization of the paper} \label{intr str}

The proof of~Theorem~\ref{theo} is self-contained. It is divided into two parts and the appendix.

In the first part~\S\ref{es} we study the equivariant slice. In~\S\ref{tcb} we discuss the
general construction of twisted cotangent bundle of a homogeneous space of a complex linear
algebraic group and study the left and two-sided actions on this affine bundle. This part is
self-contained. It is not directly related to our problem and can be read independently.
In~\S\ref{id} we give an equivalent description for $\scrX$ as a twisted cotangent
bundle of a $\t{G}$-homogeneous space. Then we reduce the question of counting
$\mathrm{Irr} \: \Lambda_{\scrX}$ to counting the {\em relevant} orbits in this homogeneous space.
Furthermore, we identify two symplectic structures on $\scrX$ (one coming from the twisted
cotangent bundle (see~\S\ref{tcb}) and another one constructed in~\cite{l2}).
Then we compute the number of relevant orbits using Bruhat decomposition. 

In the second part~\S\ref{ls} we study the linear side, which is much more simple.
In~\S\ref{ls rel} we describe $\La_{\scrX^\vee}$ in terms of conormal bundles to relevant
orbits in some cotangent bundle. In~\S\ref{ls count} we compute the number of such orbits. 

In the Appendix~\S\ref{ap} we discuss the existence of connected linear algebraic subgroups
with given Lie subalgebras (see~\S\ref{ap sub}) and give explicit forms of some objects related
to the Lie algebra $\g_2$ and arising in this article (see~\S\S\ref{ap g2},\ref{ap 7rep}).
More precisely, in~\S\ref{ap g2} we give an explicit construction of $\g_2$ and its root system,
and in~\S\ref{ap 7rep} we study the $7$-dimensional fundamental representation of $G_2$ and
compute explicitly the non-degenerate bilinear symmetric $G_2$-invariant form on this
representation.

\paragraph{Acknowledgments.} The author is grateful to his advisor M.~Finkelberg for very
fruitful discussions. This work	 is supported in part by the Moebius Contest Foundation for Young Scientists.

\section{Equivariant Slice} \label{es}


\subsection{Twisted cotangent bundle of a homogeneous space} \label{tcb}

Let $G$ be a connected linear algebraic group over $\C$ with Lie algebra $\g$, let $H \sb G$ be its closed connected subgroup with Lie algebra $\h$, let $\psi \in \h^\ast$ be an $H$-invariant covector. In this section we explicitly construct twisted cotangent bundle $T_\psi^\ast(G/H)$, associated with $\psi$. Then we study the Hamiltonian right $Q$-action on $T_\psi^\ast(G/H)$, where $Q \sb G$ is a connected closed subgroup, and the zero level $\mu_Q^{-1}(0)$ of the moment map $\mu_Q$. Finally, we study two-sided Hamiltonian actions and the zero levels of their moment maps.

In more detail, in \S\ref{tcb constr} we consider the right Hamiltonian $H$-action on the cotangent bundle $T^\ast G$. Let $\mu_H \colon T^\ast G \to \h^\ast$ be its moment map. We construct the geometric quotient of $\mu_H^{-1}(\psi)$ with respect to the right $H$-action. Next we prove that the resulting variety is smooth symplectic. In \S\ref{tcb la} we study the left $G$-action on $T_\psi^\ast(G/H)$. Then we consider the zero level $\mu_Q^{-1}(0) \sb T_\psi^\ast(G/H)$ and prove that this variety is the union of conormal bundles to so-called \emph{relevant} orbits of the left $Q$-action on $G/H$. In \S\ref{tcb 2sided} we add to the left $Q$-action on $T_\psi^\ast(G/H)$ the right $S$-action, where $S \sb N_G(H)$ is a connected closed subgroup, then show that this action is Hamiltonian, and study the relation between the zero levels of the moment maps with respect to the Hamiltonian $Q$- and $Q \times S$-actions on $T_\psi^\ast(G/H \rtimes S)$ and $T_\psi^\ast(G/H)$ respectively.


\subsubsection{Construction of $T_\psi^\ast(G/H)$} \label{tcb constr}

Recall that for any smooth algebraic variety $X$ its cotangent bundle $T^\ast X$ has a natural symplectic structure, defined as follows. The tangent bundle $T_{(x, \a)}(T^\ast X)$ can be viewed as $T_x X \oplus T^\ast_\a X$. One can define Liouville $1$-form $\la$ on $T^\ast X$ as follows: $\la_{(x, \a)}(v, \b) = \b(v)$. In the analytic setup in local coordinates $(p_1, \ld, p_n, q_1, \ld, q_n)$ on $T^\ast X$, where $(p_1, \ld, p_n)$ are local coordinates on $X$, and $(q_1, \ld, q_n)$ corresponds to coefficients at $dp_1, \ld, dp_n$ respectively, we have $\la = p_1 dq_1 + \ld + p_n dq_n$. Set $\om = d \la$ (note that $\om = dp_1 \w dq_1 + \ld + dp_n \w dq_n$ in local coordinates). It is a natural symplectic form on $T^\ast X$. If we identify $T_{(x, \a)} (T^\ast X)$ with $T_x X \oplus T_x^\ast X$ in a natural way, then $\om((v_1, \b_1), (v_2, \b_2)) = \b_1(v_2) - \b_2(v_1)$ under this identification.

Suppose that $X$ is an $H$-manifold, where $H$ acts on the right. Let $a \colon X \times H \to X$ be the action map. Then for any $x \in \h$ one can define the vector field $u_{X, x}^R$ on $X$ as follows: $u_{X, x}^R(p) = (d_{(p, 1)} a)(0, x)$. In other words, it is the velocity field of the right $H$-action with respect to $x \in \h$.

For a vector field $v$ on $X$ one can define a regular fiber-wise linear function $\s_1(v)$ on $T^\ast X$ as follows: $\s_1(v)(p, \a) = \a(v(p))$, where $v(p) \in T_p X$ is the value of $v$ at a point $p$.

\begin{fact}[{\cite[Proposition 1.4.8]{cg}}] \label{tcb 1}
	Let $X$ be a $G$-manifold, where $G$ acts on the right. Then the natural right $G$-action on $T^\ast X$ is Hamiltonian with Hamiltonian $x \mapsto H_x^R = \s_1(u_{X, x}^R)$.
\end{fact}

Recall that for a Hamiltonian right $H$-action on a smooth variety $M$, one defines the moment map $\mu_H \colon M \to \g^\ast$ as follows: $\mu_H(m) \colon x \mapsto H_x^R(m)$, where $H_x$ is the Hamiltonian corresponding to $x \in \h$. Since $H$ is connected, the moment map $\mu_H$ is
$H$-equivariant~\cite[Lemma 1.4.2]{cg}.

If $X = G$, then the right $H$-action on $T^\ast G$ is Hamiltonian. Consider the corresponding moment map $\mu_H \colon T^\ast G \to \g^\ast$ and put $Y = \mu_H^{-1}(\psi)$. Explicitly, $Y = \{(g, \a) \in T^\ast G : \a(u_{G, x}^R(g)) = \psi(x)\}$. One can easily see that the natural projection $Y \to G$ is an affine bundle with $(\dim G - \dim H)$-dimensional fibers.

Recall that the left $G$-action on $G$ (and hence on $T^\ast G$) commutes with the right $H$-action.

\begin{lem} \label{tcb 2}
$Y = \mu_H^{-1}(\psi)$ is invariant under the left $G$-action.
\end{lem}

\begin{proof}
The idea of the proof is that $G$- and $H$-actions commute, so $G$-action preserves $u_{G, x}^R$ for any $x \in \h$ and thus preserves $Y = \mu_H^{-1}(\psi)$. More precisely, let $q \in G$, $(g, \a) \in Y$. By definition of $Y$ we have $\a(u_{G, x}^R(g)) = \psi(x)$. Set $L_q \colon G \to G$, $g \mapsto qg$. Then $(d_g L_q)^\ast \colon T_{qg}^\ast G \to T_{g}^\ast G$, so $q \colon (g, \a) \mapsto (qg, (d_{qg} L_{q^{-1}})^\ast(\a)) =: (qg, \t{\a})$. Since $G$- and $H$-actions commute, $(d_g L_q)(u_{G, x}^R(g)) = u_{G, x}^R(qg)$. So we have 
\[\t{\a}(u_{G, x}^R(qg)) = (d_{qg} L_{q^{-1}})^\ast(\a)[d_g L_q(u_{G, x}^R(g))] = \a(d_{qg} L_{q^{-1}} \circ d_g L_q(u_{G, x}^R(g))) = \a(u_{G, x}^R(g)) = \psi(x),\]
and we are done.
\end{proof}

Since $\mu_H$ is $H$-equivariant and $\psi$ is $H$-invariant, $Y$ is invariant under the right
$H$-action.

\begin{defi}[{\cite[\S4.2]{pv}}] \label{tcb 3}
Let $G$ be a connected linear algebraic group, $X$ be a $G$-variety. Then the variety $Y$ together with a morphism $\pi \colon X \to Y$ is called \emph{a geometric quotient}, if 
\begin{itemize}
\item[(1)] $\pi$ is surjective and its fibers are exactly $G$-orbits on $X$;
\item[(2)] $Y$ has a quotient topology;
\item[(3)] $\O_Y = \pi_\ast(\O_X^G)$, i.e. $\O_Y(U) = \O_X(\pi^{-1}(U))^G$ for any open $U \sb Y$.
\end{itemize}
\end{defi}

\begin{rem}[{\cite[Theorem~4.2]{pv}}] \label{tcb 4}
If $Y$ is normal and $X$ is irreducible, it is enough to require only (1) in Definition \ref{tcb 3}.
\end{rem}

\begin{fact}[{Rosenlicht, \cite[Theorem~4.4]{pv}}] \label{tcb 5}
If $X$ is any irreducible $G$-variety, then there exists $G$-stable non-empty open subset $U \sb X$, which has a geometric $G$-quotient.
\end{fact}

\begin{rem}
There is a notion of so called \emph{rational quotient} (see \cite[\S4.2]{pv}). The original theorem of Rosenlicht (see \cite[Theorem 2]{r}) concerns the existence of rational quotients. However, using \cite[Theorem~2.3, Theorem 4.2]{pv}, and rational quotients, one can quite easily prove Fact~\ref{tcb 5}.
\end{rem}

Now we return to our initial notation. Then next result is the most important in this subsection.

\begin{prop} \label{tcb 6}
Let $H$ act on $\mu_H^{-1}(\psi)$ on the right. Then there exists a smooth geometric $H$-quotient.
\end{prop}

\begin{rem}
	The statements of Prop.~\ref{tcb 6} and Lemma~\ref{tcb 8} are very similar to parts i) and ii) respectively of a special case of Marsden-Weinstein-Meyer theorem (see \cite[Theorem~9.4]{h}). In \S\ref{tcb la} we prove (see Lemma~\ref{tcb 9}) that the natural left $G$-action on $\mu_H^{-1}(\psi)/H$ is Hamiltonian. So we get an example of so-called \emph{affine Hamiltonian Lagrangian $G$-bundle} studied in details in \cite{l} in the holomorphic category. In particular, the geometric quotient $\mu_H^{-1}(\psi)/H$ is introduced in \cite[\S(2.4)]{l}, see also the discussion in \cite[\S2.3]{c}. However, the proof of Prop.~\ref{tcb 6} we give is original.
\end{rem}

\begin{proof}
Denote by $Z$ the set of $H$-orbits in $Y = \mu_H^{-1}(\psi)$. There is a natural surjective map $\pi \colon Y \to Z$. Then one can equip $Z$ with the quotient topology and a sheaf of $\C$-algebras $\O_Z = \pi_\ast(\O_{T^\ast G})^H$. We claim that $(Z, \O_Z)$ is a geometric quotient of $Y$. The only thing we need to check is that the ringed space $(Z, \O_Z)$ is a variety. First we show that any point $z \in Z$ has an open neighborhood that is isomorphic as a ringed space to some affine variety. Second we will show that $(Z, \O_Z)$ is separated. Finally, we will show that $Z$ is smooth. \newline

By Rosenlicht theorem there is a non-empty open $H$-stable subset $U \sb Y$, which has a geometric quotient. It means that $(\pi(U), \O_Z |_{\pi(U)})$ is a variety, so any point $z \in \pi(U)$ has an open neighborhood isomorphic as a ringed space to an affine variety. Since $G$- and $H$-actions commute,  for any $H$-stable $U \sb Y$ that has a geometric $H$-quotient, and any $g \in G$, the open subset $g \cd U \sb Y$ also has a geometric $H$-quotient. Without loss of generality we can assume $(\se, \a) \in U$, where $\se$ stands to the neutral element of $G$.

Take a point $z \in Z$, let $y = (g, \b) \in \pi^{-1}(z)$. Since $G$ acts transitively on $G$ on the left, then without loss of generality $g = \se$. Note that $T_\se^\ast(G/H) \cong \Ann_{\g^\ast}(\h)$ in a natural way. There is a natural projection $p \colon G \to G/H$. Since $(\a - \b)(u_{G, x}^R) = 0$ for any $x \in \h$, one can easily find an open subset $V \sb G/H$ such that $[\se] \in V$, and an algebraic $1$-form $\nu$ on $V$ such that $p^\ast(\nu) |_\se = \a - \b$. Set $\t{\nu} = p^\ast(\nu)$; it is an $H$-invariant $1$-form on $p^{-1}(V)$. Also $\t{\nu} + Y  = Y$ over $V$, because any $H$-orbit goes to a single point under the projection $p$, so adding $\t{\nu}$ does not change the value of the 1-form on the tangent vector to an $H$-orbit. Without loss of generality we may suppose $U \sb p^{-1}(V)$. Since $U$ and $U - \t{\nu}$ are isomorphic as $H$-varieties, $U - \t{\b}$ also has a geometric quotient. But now we have $y \in U - \t{\nu}$. So $y$ lies in an open subset possessing a geometric quotient. Thus $z$ has an open neighborhood isomorphic to an affine variety. \newline

Second, let us show that $Z$ is separated. We need to prove that the diagonal $\De_Z \sb Z \times Z$ is closed. Suppose $(z_1, z_2) \sb \o{\De}_Z$ such that $z_1 \ne z_2$. There is a projection $\rho \colon Z \to G/H$, coming from the natural projection $Y \to G/H$. First consider the case $\rho(z_1) \ne \rho(z_2)$. Since $G/H \times G/H$ is a variety, $V = G/H \times G/H - \De_{G/H}$ is open and contains $(\rho(z_1), \rho(z_2))$. Its preimage under $\rho \times \rho$ does not contain $\De_Z$, contains $(z_1, z_2)$ and is open.
So the case $\rho(z_1) \ne \rho(z_2)$ is impossible.

Now consider the case $\rho(z_1) = \rho(z_2)$. Without loss of generality $z_1  = ([\se], \a_1)$, $z_2 = ([\se], \a_2)$. Denote $D = \o{\De}_Z - \De_Z$. As was shown in the first part of the proof, for any $\nu' \in \Ann_\g(\h)$ we may suppose that there is an $H$-invariant $1$-form $\t{\nu}$ on $U$ such that $\t{\nu}|_\se = \nu'$ and $\t{\nu} + Y = Y$ over $U$. Thus $\t{\nu} + U$ is another open neighborhood, which has a geometric quotient. So $(([\se], \a_1 + \nu'), ([\se], \a_2 + \nu')) \in D$. Thus $D$ contains $\{(([\se], \b_1), ([\se], \b_2)) : \b_2 - \b_1 = \a_2 - \a_1\}$. Such a subset has the same dimension as an affine fiber of $Y$ over $G/H$, i.e. $\dim G - \dim H$. Varying the first coordinate over all $G/H$, we acquire the additional dimension $\dim G/H$. So $\dim D \br 2(\dim G - \dim H) = \dim Z = \dim \De_Z$. This contradicts $D = \o{\De}_Z - \De_Z$. \newline

Finally, let us show that $Z$ is smooth. The set of regular points of any variety is non-empty and open. Since $G$ acts transitively on the set of fibers of the projection $Z \to G/H$, it is enough to prove that any two points $z_1, z_2 \in Z$ in the same fiber over $[\se]$ have isomorphic open neighborhoods. As was shown in the first part of the proof, we may assume that there is an $H$-invariant $1$-form $\t{\nu}$ on $U$ such that $\t{\nu}|_\se + z_1 = z_2$ and $\t{\nu} + Y = Y$ over $U$. Clearly adding such a form $\t{\nu}$ gives an isomorphism of geometric $H$-quotients, so $z_2$ and $z_1$ have isomorphic open neighborhoods.
\end{proof}

We denote the resulting geometric quotient of $Y = \mu_H^{-1}(\psi)$ by $T_\psi^\ast(G/H) = \scrX$
and call it the $\psi$-\emph{twisted cotangent bundle} of $G/H$. We have the natural projection $\pi \colon Y \to T_\psi^\ast(G/H)$.

\begin{ex} \label{tcb 7}
If $\psi = 0$, then there is a natural isomorphism $\scrX \cong T^\ast(G/H)$. 
\end{ex}

\begin{proof}
In case $\psi = 0$ we have $Y = \mu_H^{-1}(0) = \{(g, \a) : \a(u_{G, x}^R(g)) = 0 \: \forall x \in \h\}$. Note that $Y$ is a locally trivial vector bundle over $G$. For any fixed $g$ there is a natural pairing of the fiber of $Y$ over $g$ and $T_{[g]}(G/H)$. So there is a natural surjection $F \colon Y \to T^\ast(G/H)$. If we suppose that $F$ is regular, then by~Definition~\ref{tcb 3}(3), the morphism $F$ induces a regular bijective map $\t{F} \colon \scrX \to T^\ast(G/H)$. Clearly $Y$ is irreducible. Since $T^\ast(G/H)$ is smooth, it is normal. Since $\char \C = 0$ and $\t{F}$ is bijective, it is birational by \cite[Theorem~5.1.6]{s}. By Zariski main theorem~\cite[Theorem~5.2.8]{s} $\t{F}$ is an isomorphism and we are done.

It remains to show that $F$ is regular. Note that $T^\ast(G/H)$ is locally trivial vector bundle. Hence for any point $z \in G/H$ there is an open subset $V \ni z$ such that $T^\ast(G/H)$ over $V$ is trivial, i.e.~there are $d = \dim (G/H)$ regular $1$-forms $\a_1, \ld, \a_d$ on $V$ such that any regular $1$-form on $V$ is of the form $f_1 \a_1 + \ld + f_d \a_d$, where $f_j \in \O_{G/H}(V)$. Let $\b_j$ be the pullback of the regular $1$-form $\a_j$ under the projection $p \colon G \to G/H$. Then any section of $Y$ over $p^{-1}(V)$ can be presented in the form $g_1 \b_1 + \ld + g_d \b_d$, where $g_j \in \O_G(p^{-1}(V))$, that implies $Y |_{p^{-1}(V)} \cong V \times \C^d$. So $F$ looks like $p \times \Id \colon p^{-1}(V) \times \C^d \to V \times \C^d$ in the chosen open neighborhoods. Clearly such a map is regular.
\end{proof}

\begin{lem} \label{tcb 8}
The natural symplectic structure of $T^\ast G$ descends to the variety $\scrX = T_\psi^\ast(G/H)$.
\end{lem}

\begin{proof}
  Recall that $T^\ast G$ carries the natural symplectic form $\om$. For any $y \in Y$ denote by $\bO_y$ the $H$-orbit through $y$. Then $T_{\pi(y)} \scrX \cong T_yY / T_y \bO_y$. To show that the $2$-form $\om_y$
  descends to the $2$-form $\t{\om}_y$ on $T_yY / T_y \bO_y$ it is enough to check that for any $v \in T_y Y, w \in T_y \bO_y$ we have $\om_y(v, w) = 0$. Note that $T_y \bO_y = \{u_{T^\ast G, x}^R(y) : x \in h\}$, so $w = u_{T^\ast G, x}^R(y)$ for some $x \in \h$. We have $\om_y(\cd, u_{T^\ast G, x}^R(y)) = d_y H_x^R$ by definition of Hamiltonian action (cf. \cite[Definition~1.4.1, \S 1.2]{cg}). So we need to show that $d_y H_x^R(v) = 0$ for any $v \in T_y Y$. But $H_x^R |_Y \equiv \psi(x) = const$, so $d_y H_x^R  = 0$ on $T_y Y$. So we actually get some $2$-form $\t{\om}$ on $\scrX$. \newline

  Let us show that $\t{\om}$ is non-degenerate. Let $y = (g, \a) \in T_y Y$. We have a natural
  identification $T_y (T^\ast G)\cong T_g G \op T_g^\ast G$. Consider any $(v, \b) \in T_y Y$. First suppose that $v$ does not lie in the tangent space to the right $H$-orbit on $G$ passing through the point $g$. It means that $v \notin \Span_{x \in \h}(u_{G, x}^R)$, so one can find $\ga$ in the tangent space to the fiber of $Y$ over $g$ (that equals $\{\ga \in T_g^\ast G : \ga(u_{G, x}^R(g)) = 0 \text{ for any $x \in \h$}\}$) such that $\ga(v) \ne 0$. So $\om((v, \b), (0, \ga)) \ne 0$. Second, consider the case when $v \in \Span_{x \in \h}(u_{G, x}^R(g))$. In this case there is $w \in T_y \bO_y$ such that its projection to $T_g G$ equals $v$. Hence we may assume that $v = 0$, because adding vectors from the tangent space to the $H$-orbit does not change the value of the $2$-form. Since the natural projection $T_y Y \to T_g G$ is clearly surjective, one can take $(v', \b') \in T_y Y$ such that $\b(v') \ne 0$ and $\b'$ is a covector. In this case also $\om((0, \b), (v', \b')) \ne 0$. Therefore, $\om$ is non-degenerate. \newline

	It remains to show that $\t{\om}$ is regular and closed. One can consider the Liouville 1-form $\t{\la}$ on $\scrX$ (similar to $\la$ on $T^\ast G$). Indeed, if $y = (g, \a)$, then $T_{\pi(y)} \scrX \cong T_{[g]}(G/H) \oplus \{\b \in T_g^\ast G : \b(u_{G, x}^R) = 0\}$. Thus there is a natural pairing of $T_{[g]}(G/H)$ with $\{\b \in T_g^\ast G : \b(u_{G, x}^R) = 0\}$, that allows us to construct the $1$-form $\t{\la}$. If we suppose that $\t{\la}$ is regular, then $\pi^\ast(\t{\la}) = \la |_Y$, so $\pi^\ast(d \t{\la}) = d \pi^\ast(\t{\la}) = d \la |_Y = \om |_Y$, and hence $d\t{\la} = \t{\om}$. Therefore, $\t{\om}$ is algebraic and closed.

	It remains to show that $1$-form $\t{\la}$ is regular. Any point of $G/H$ has an open neighborhood $V$ such that both $T^\ast(G/H)$ and $T(G/H)$ are trivial over $V$. Then there are $d = \dim(G/H)$ base sections $\a_1, \ld, \a_d$ (resp.~$v_1, \ld, v_d$) of $T^\ast(G/H)$ (resp.~$T(G/H)$) over $V$. Then any section of $T \scrX$ over $U = \rho^{-1}(V)$ (where $\rho \colon \scrX \to G/H$ is the natural projection) is of the form $(g_1 v_1 + \ld + g_d v_d, f_1 \a_1 + \ld + f_d \a_d)$ ($g_j, f_j \in \O_\scrX(U)$). Indeed,  we can present $T_z \scrX$, where $z = ([g], \b)$, as a direct sum $T_{[g]}(G/H) \oplus \{\b \in T_g^\ast G : \b(u_{G, x}^R) = 0\}$. So $T \scrX |_U \cong \C^d \times \C^d$, and in such local coordinates
\[\t{\la} \colon (g_1 v_1 + \ld + g_d v_d, f_1 \a_1 + \ld + f_d \a_d) |_z \mapsto \sum \l_{j = 1}^d f_j(z) g_j(z)\]
so $\t{\la}$ clearly is a regular $1$-form.
\end{proof}

In the next subsection we still denote the natural symplectic form on $T_\psi^\ast(G/H)$ by $\t{\om}$.


\subsubsection{Left action on $T_\psi^\ast(G/H)$} \label{tcb la}

Recall that there is a natural left $G$-action on $T_\psi^\ast(G/H)$ inherited from the left $G$-action on $\mu_H^{-1}(\psi)$. For any variety $X$ with a left $G$-action, and $x \in \g$ denote by $u_{X, x}^L$ the velocity field of this left $G$-action with respect to $x$. Note that for any $y \in Y = \mu_H^{-1}(\psi)$, the tangent vector $u_{T^\ast G, x}^L(y) = u_{Y, x}^L(y)$ lies in $T_y Y$, since $Y$ is $G$-stable. Similarly to~\S\ref{tcb 1}, the left $G$-action on $T^\ast G$ is Hamiltonian with Hamiltonian $H_x^L = - \s_1(u_{G, x}^L) \in \O(T^\ast G)$. Here we need the negative sign, because the map 
\[\g \to \{\text{symplectic vector fields on $T^\ast G$}\}, \quad x \mapsto \text{velocity field of left $G$-action with respect to $x$}\]
is an anti-homomorphism of Lie algebras, while the same map for right $H$-action is a homomorphism of Lie algebras.

\begin{lem} \label{tcb 9}
For any $x \in \g$, the function $H_x^L$ is $H$-invariant, and the left $G$-action on $\scrX = T_\psi^\ast(G/H)$ is Hamiltonian with Hamiltonian $x \mapsto \pi_\ast(H_x^L |_Y)$.
\end{lem}

\begin{proof} 
  Since the right $H$-action and the left $G$-action on $T^\ast G$ commute, $u_{G, x}^L$ is
  $H$-invariant. Hence $H_x^L = - \s_1(u_{G, x}^L)$ defined via $u_{G, x}^L$ is also $H$-invariant.
  So we get a regular function $\pi_\ast(H_x^L)$ on $\scrX$, its value on an orbit as a point is
  equal to the value of $H_x^L$ at any point of this orbit. Note that regularity of function
  $\pi_\ast(H_x^L)$ follows from the property~(3) of~Definition~\ref{tcb 3}.

  Now consider any $z \in T_\psi^\ast(G/H)$, $y \in \pi^{-1}(z)$, $x \in \g$. Consider a velocity
  vector field $u_{\scrX, x}^L$ of the left {$G$-action} on $\scrX$ with respect to $x$. Clearly
  $u_{\scrX, x}^L(z) = d_y \pi (u_{Y, x}^L(y))$. We need to show
  $\t{\om}(\cd,  u_{\scrX, x}^L) = d_z \pi_\ast(H_x^L)$ on $T_z \scrX$, i.e.\
  $\t{\om}(v, u_{\scrX, x}^L) = (d_z \pi_\ast(H_x^L))(v)$ for any $v \in T_z \scrX$.
  It is equivalent to $\om(w, u_{Y, x}^L) = (d_y H_x^L)(w)$ for any $w \in T_y Y$. Indeed, if
  $v = (d_y \pi)(w)$, then $(d_z \pi_\ast(H_x^L))[(d_y \pi)(w)] = d_y(\pi^\ast \pi_\ast(H_x^L))(w) = (d_y H_x^L)(w)$.
  Since the left $G$-action on $T^\ast G$ is Hamiltonian, the equality
  $\om(w, u_{Y, x}^L) = (d_y H_x^L)(w)$ holds true for any $w \in T_y Y$ by the definition of
  Hamiltonian action (see \cite[Definition 1.4.1, \S1.2]{cg}).
\end{proof}

Let $Q \sb G$ be a closed connected subgroup of $G$ with Lie algebra $\q$. Clearly the left
$Q$-action on $T_\psi^\ast(G/H)$ is Hamiltonian with Hamiltonian $x \mapsto \pi_\ast(H_x^L)$.
Denote by $\mu_Q \colon T_\psi^\ast(G/H) \to \q^\ast$ the moment map. 

Note that there is a natural bijection from the set of $Q$-orbits on $G/H$ to the set of $H$-orbits on $Q \bs G$.

\begin{defi}[{cf.~\cite[\S2.2]{fgt}}] \label{tcb 10}
  An $H$-orbit $\bO$ on $Q \bs G$ is called \emph{relevant}, if $\psi |_{\Lie \Stab_H (p)} = 0$
  for any point $p \in \bO$. A $Q$-orbit in $G/H$ is called \emph{relevant}, if the
  corresponding $H$-orbit in $Q \bs G$ is relevant.
\end{defi}

\begin{rem} 
It is enough to check $\psi |_{\Lie \Stab_H (p)} = 0$ for any one point $p \in \bO$, because $\Lie \Stab_H (p)$ for all $p \in \bO$ are $H$-conjugate and $\psi$ is $H$-invariant.
\end{rem}

Set $\La_\scrX = \mu_Q^{-1}(0)$. Denote by $Y_L$ the zero level of the moment map
$T^\ast G \to \q^\ast$ for the left $Q$-action on $T^\ast G$. Explicitly, $Y_L = \{(g, \a) : \a(u_{G, x}^L(g)) = 0 \: \forall x \in \q\}$. As before, $Y = Y_R = \mu_H^{-1}(\psi)$. Since Hamiltonians
for the $Q$-action on $\scrX$ descend from the ones for $Q$-action on $T^\ast G$
by~Lemma~\ref{tcb 9}, we have $\La_\scrX = \pi(Y_L \cap Y_R)$. Note that $Y_L \cap Y_R$ is both
$G$- and $H$-invariant. Recall that there is a natural projection $\rho \colon Y \to G/H$. 

\begin{lem} \label{tcb 11}
The fiber $\rho^{-1}(g) \cap (Y_L \cap Y_R)$ is non-empty iff $[g] \in G/H$ lies in a relevant $Q$-orbit.
\end{lem}

\begin{proof}
  The fiber $\rho^{-1}(g) \cap (Y_L \cap Y_R)$ as a subset of $T^\ast_g G$ is cut out by two systems
  of linear equations: $\a(u_{G, x}^R(g)) = \psi(x)$ for any $x \in \h$, and
  $\a(u_{G, x}^L(g)) = 0$ for any $x \in \q$. Set $K_L = \Span_{x \in \q}(u_{G, x}^L(g))$, $K_R = \Span_{x \in \h}(u_{G, x}^R(g))$. Then our two systems have a common solution iff $\psi$ vanishes on all $x$ such that $u_{G, x}^R(g) \in K_L$. Consider the $H$-orbit $\bO \sb Q \bs G$, containing the image of $g$. There is a natural map $\tau \colon H \to \bO$, $h \mapsto [gh]$. Also there is a natural map $q \colon G \to Q \bs G$, so $\tau = q \circ \t{\tau}$, where $\t{\tau} \colon H \to G$, $h \mapsto gh$. Thus 
\[\Lie \Stab_H([g]) = \ker (d_\se \tau) = \{x \in \h : d_\se \t{\tau}(x) \in \ker d_g q\} = \{x \in \h : u_{G, x}^R(g) \in K_L\}\]
So the fiber over $g$ is non-empty iff $\psi$ vanishes on $\Lie \Stab_H([g])$, that
is equivalent to the relevancy of $\bO$ by~Definition~\ref{tcb 10}.
\end{proof}

The following corollary is the key result of this subsection.

\begin{cor} \label{tcb 12} 
  The subvariety $\La_\scrX$ is a union of twisted conormal bundles over the relevant $Q$-orbits in
  $G/H$. In particular, if there are only finitely many relevant $Q$-orbits in $G/H$,
  then $\La_\scrX$ is of pure dimension $\dim(G/H)$.
\end{cor}

The word "twisted" in~Corollary~\ref{tcb 12} means that we deal with some affine bundle over
a relevant orbit such that the linearization of this bundle is the conormal bundle over corresponding orbit. 

\begin{proof}
  It follows immediately from Lemma~\ref{tcb 11} and $\La_\scrX = \pi(Y_L \cap Y_R)$ that $\La_\scrX$ is a union of some affine bundles over relevant $Q$-orbits. Consider any relevant $Q$-orbit $\bO \sb G/H$. To see that after linearization we get a conormal bundle over $\bO$ (as a subbundle in $T^\ast(G/H)$), notice that this linearization lies in $T^\ast(G/H)$ (cf. Example \ref{tcb 7}) and is given by linear equations $\a(u_{G, x}^L(g)) = 0$ for any $x \in \q$, where $[g] \in \bO$. Since the images of $u_{G, x}^L(g)$ under the natural projection $G \to G/H$ span the tangent space to $\bO$ at $[g]$,
  we get exactly the conormal bundle to $\bO$.
\end{proof}

\begin{rem} \label{rem Los-rel-orb}
	The analogous statement for the non-twisted bundles is well-known (cf.~\cite[Exercises 10.6]{l3}). Corollary~\ref{tcb 12} formally implies Exercise 10.6 only for cotangent bundles of homogeneous spaces, and the general case mentioned in Exercise 10.6 follows immediately from explicit formula for Hamiltonians (see Fact~\ref{tcb 1}).
\end{rem}

\begin{cor} \label{tcb 13}
If $Q$ is solvable and there are only finitely many relevant $Q$-orbits in $G/H$, then $\La_\scrX$ is Lagrangian and the number of its irreducible components coincides with the number of relevant orbits.
\end{cor}

\begin{proof}
Since $\La_\scrX$ is coisotropic in this case (see \cite[Theorem 1.5.7]{cg}), the corollary immediately follows from Corollary \ref{tcb 12}.
\end{proof}


\subsubsection{Two-sided action on $T_\psi^\ast(G/H)$} \label{tcb 2sided}

Suppose that there is a closed connected subgroup $S \sb G$ with Lie algebra $\fs$
such that $S \sb N_G(H)$, i.e.~$php^{-1} \in H$ for any $h \in H$, $p \in S$. Suppose that $\psi$
vanishes on $[\h, \fs]$. Then $\psi$ vanishes on $\Ad_p(x) - x$ for any $x \in \h, p \in S$. Indeed, consider an exponential map $t \mapsto \exp(t y)$ (see \cite[\S 3.1]{k}) with initial velocity $y \in \fs$. We have the $\bG_a$-orbit $\{\Ad_{\exp(ty)(x)} : t \in \C\}$ in $\h$ through $x$, with the natural $\bG_a$-action $s \cd \Ad_{\exp(ty)(x)} = \Ad_{\exp((t+s)y)(x)}$. Since $[x, y] = 0$, the differential of $\bG_a$-action map $s \mapsto s \cd x = \Ad_{\exp(sy)}(x)$ vanishes, and our $\bG_a$-orbit is a single point $\{x\}$. So for any $y \in \fs$, $t \in \C$ holds $\Ad_{\exp(ty)}(x) = x$. Since $\exp(t y)$ for all possible $y \in \h, t \in \C$ generate $S$, we have $\Ad_p(x) = x$ for any $x \in \h, p \in S$.

Also suppose that $H \cap S = \{\se\}$, so $\h \cap \fs = 0$.

\begin{lem} \label{tcb 14}
	The right $S$-action on $T^\ast G$ preserves $Y = \mu_H^{-1}(\psi)$.
\end{lem}

\begin{proof}
  Since the left $G$-action and the right $S$-action commute, and the left $G$-action
  preserves $Y$ by Lemma \ref{tcb 2}, it is enough to consider only the fiber of $Y$ over the neutral element $\se$.
  Let $(\se, \a) \in Y$, where $\a \in T_\se^\ast G$. It is equivalent to $\a(x) = \psi(x)$ for
  any $x \in \h$. Next we apply the right translation by $p$ to $(\se, \a)$ and get another point
  $(p, \b)$. We need to check that $\b(u_{G, x}^R(p)) = \psi(x)$ for any $x$. Since the right and
  left $G$-actions commute, it is equivalent to the equality $\Ad_p^\ast(\a)(x) = \psi(x)$ for
  any $x$ and $p \in S$. So it is enough to show that $\a(\Ad_p(x) - x) = 0$ for any $x \in \h$.
  Indeed, since $p H p^{-1} \sb H$, we have $\Ad_p(x) - x \in \h$, and $\psi$ vanishes on
  $\Ad_p(x) - x$ by our assumptions.
\end{proof}

Since $S \sb N_G(H)$, the right $S$-action on $Y$ induces the right $S$-action on $\scrX$.

\begin{lem}[cf. Lemma \ref{tcb 9}] \label{tcb 15}
	For any $x \in \fs$, the function $H_x^R |_Y$ is $S$-invariant and the right $S$-action on $\scrX$ is Hamiltonian with Hamiltonian $x \mapsto \pi_\ast(H_x^R|_Y)$. 
\end{lem}

\begin{proof} 
  The proof is completely similar to the proof of Lemma \ref{tcb 9}. The only difference is
  that we restrict our functions to $Y$ and use the additional condition
  $\psi |_{[\fs, \h]} = 0$ as in Lemma \ref{tcb 14}.
\end{proof}

Clearly the left $Q$- and right $S$- actions on $\scrX$ commute, because the left and right
translations on $G$ commute. Hence there is a natural $Q \times S$-action on $\scrX$. Put
$R = Q \times S$. Since both $P$- and $Q$- actions are Hamiltonian, the natural $R$-action
is Hamiltonian with Hamiltonian $(x_l, x_r) \mapsto \pi_\ast(H_{x_l}^L |_Y + H_{x_r}^R |_Y)$,
where $x_l \in \q$, $x_r \in \fs$ and $\Lie R = \q \oplus \fs$. Set $\t{H} = S \ltimes H \sb G$,
it is a closed connected subgroup. The next proposition is the key point in the proof
of Theorem~\ref{theo}.

\begin{prop} \label{tcb 16}
  Suppose that the subgroups $S$, $Q$ are solvable, $\psi$ extends to the $\t{H}$-invariant covector in $(\Lie \t{H})^*$, and $Q$ has only finitely many relevant orbits in $G/\t{H}$. Then
	\begin{itemize}
		\item[(i)] the zero level $\La_1$ of the moment map $T_\psi^\ast(G/H) \to (\Lie R)^\ast$.
		\item[(ii)] the zero level $\La_2$ of the moment map $T_\psi^\ast(G/\t{H}) \to (\Lie Q)^\ast$.
	\end{itemize}
	are both Lagrangian subvarieties. The number of irreducible components of $\La_1$ and the number
        of irreducible components of $\La_2$ are both equal to the number of relevant $Q$-orbits
        in $G/\t{H}$.
\end{prop}

\begin{proof}
  Both varieties are coisotropic by \cite[Theorem 1.5.7]{cg}. By Corollary \ref{tcb 13}, the variety
  $\La_2$ is Lagrangian and has the desired number of irreducible components. Hence it remains to prove that $\dim \La_1 = \dim \La_2 + \dim S$ and that $\La_1$, $\La_2$ have the same number of irreducible components. 
	
  First consider only the right $S$-action on $T_\psi^\ast(G/H)$ and denote its moment map
  by $\mu_S$. Clearly $\mu_S^{-1}(0)$ is the image of some irreducible affine
  bundle over $G$ under the projection $\pi \colon \mu_H^{-1}(\psi) \to T_\psi^\ast(G/H)$ discussed in detail in \S\ref{tcb constr}. So $\mu_S^{-1}(0)$ is irreducible. There is a natural bijective
  morphism $\mu_S^{-1}(0) \to T_\psi^\ast(G/\t{H})$. Its fibers are exactly $S$-orbits, so
  by Remark \ref{tcb 4} $T_\psi^\ast(G/\t{H})$ is a geometric quotient of $\mu_S^{-1}(0)$. 
	
  Under the projection $\mu_S^{-1}(0) \to T_\psi^\ast(G/\t{H})$, the variety
  $\mu_Q^{-1}(0) \cap \mu_S^{-1}(0)$ maps to $\La_2$ surjectively and each fiber is isomorphic to
  $S$. Hence we get a morphism $\phi \colon \La_1 \to \La_2$ with fibers isomorphic to $S$.
  This morphism has one important additional property. Since it comes from the geometric quotient projection, the image of any closed subset in $\La_1$, containing with each its point the whole fiber (isomorphic to $S$) through this point, is a closed subset in $\La_2$. Clearly $\dim \La_1 = \dim \La_2 + \dim S$, so $\La_1$ is Lagrangian. 
	
	Let us show that the preimage of any irreducible component $C$ of $\La_2$ is an irreducible component of $\La_1$. Indeed, suppose $\phi^{-1}(C) = C_1' \cup C_2'$, where $C_i'$ are closed subsets of $\La_1$. Since $S$ is irreducible, for any $x \in \La_2$ the fiber over $x$ is contained either in $C_1'$ or in $C_2'$. Denote by $C_i$ the set of $x \in C_i'$ such that the fiber over $\phi(x)$ is contained in $C_i'$. Then $C_i$ is closed (by semi-continuity of dimension of the fiber) and with each its point contains the whole fiber through this point. So $\phi(C_i)$ are closed in $\La_2$ and $C = \phi(C_1) \cup \phi(C_2)$, so $C = \phi(C_i)$ for some $i \in \{1, 2\}$ and hence $C_i' = \phi^{-1}(C)$. This observation implies that $\La_1$ and $\La_2$ have the same number of irreducible components.
\end{proof}


\subsection{Relation between $T_\psi^\ast$ and $\cS_e$, and the end of the proof} \label{id}

In this section we establish an equivariant isomorphism of symplectic varieties
$G \times \cS_e\iso T_\psi^\ast(G/H)$ defined in \S\ref{tcb}. So by Proposition \ref{tcb 16} we reduce the proof to
counting relevant orbits. Finally, we perform this count and complete the proof of
Theorem~\ref{theo}.

In more detail, in \S\ref{id ginz} we use a result of Gan and Ginzburg~\cite[Lemma~2.1]{gg} and
establish an isomorphism of algebraic varieties $G \times \cS_e\iso T_\psi^\a(G/H)$. Also we
relate the $G \times S$-action (where $S \sb G$ is a connected closed subgroup with some
additional properties) on $G \times \cS_e$ and the two-sided $G \times S$-action on
$T_\psi^\ast(G/H)$. In \S\ref{id losev} we show that the symplectic structure on $G \times \cS_e$ constructed in \cite{l2} coincides with the symplectic structure coming from the twisted cotangent bundle. Also we give an alternative proof (see Lemma \ref{id_los 2}) of the fact that $G \times \SL(2)$-action defined in \S\ref{intr form} is Hamiltonian. Finally, in \S\ref{id ro} we count relevant orbits and complete the proof of~Theorem~\ref{theo}.


\subsubsection{Identification with the twisted cotangent bundle} \label{id ginz}

Let $G$ be a connected semisimple linear algebraic group over $\C$ with Lie algebra $\g$.
Let $e \in \g$ be a non-zero nilpotent element and let $(e, f, h)$ be an $\sl_2$-triple
in $\g$ (it exists by the Jacobson-Morozov Theorem~\cite[Theorem~3.7.1]{cg}). Set
$\cS_e = e + \ker \ad_f$, then $\cS_e$ is the \emph{Slodowy slice} (cf.~\cite[\S 3.7.14]{cg}).
Set $\scrX = G \times \cS_e$.

This section is mostly based on \cite[\S2]{gg}. We define $\chi \in \g^\ast$, $\chi(x) = (x, e)$,
where $(\cd, \cd)$ is the Killing form on $\g$. Recall that the Killing form is non-degenerate
and provides a non-degenerate pairing $\g_\al \ot \g_{-\al} \to \C$, where $\al \in R$ ($R$ is the
root system of $G$) or $\al = 0$. Moreover, $\g_\al \bot \g_\be$ for $\al + \be \ne 0$. 

Consider $\g$ as an $\sl_2$-module, so under the action of $\ad_h$ we have a weight decomposition
$\g = \bigoplus_{i \in \Z} \g_i$, where $\g_i = \{x \in \g : [h, x] = ix\}$. Note  on $\g_{-1}$ is
equipped with a symplectic form $\om$: $\om(x, y) = \chi([x, y]) = ([x, y], e)$. Consider an
integer $d < \f{1}{2} \dim \g_{-1}$ and $l \in \IGr(d, \g_{-1})$, i.e.~$l$ is an isotropic subspace of dimension $d$. Set
\[\g_{\mr k} = \bigoplus_{i \mr k} \g_i, \quad \n_l = l \oplus \g_{\mr -2}, \quad \m_l = l^{\bot_\om} \oplus \g_{\mr -2}\]
Since $l \sb l^{\bot_\om}$, we have $\n_l \sb \m_l$. 

\begin{lem} \label{constr Nl}
  One can choose $l \in \IGr(d, \g_{-1})$ such that there is a closed connected subgroup
  $N_l \sb G$ with Lie algebra $\n_l$.
\end{lem}

\begin{proof}
  Clearly $\n_l$ is nilpotent. By~Proposition~\ref{lie exist} it is enough to find a
  maximal torus $T \sb G$ such that~$\n_l$ is invariant under the $T$-action.
  By \cite[Lemma 3.1.4~(a)]{cg} $h$ lies in a Cartan subalgebra $\ft$. Consider the maximal torus
  $T$ corresponding to this subalgebra. Since $[x, h] = 0$ for any $x \in \ft$, we see that $h$
  is $T$-invariant, i.e.~$\Ad_t(h) = h$ for any $t \in T$. Since $\Ad_t$ commutes with the
  Lie bracket and $\Ad_t(h) = h$, we see that $\Ad_t(\g_i) = \g_i$ for any $i$. The only question
  is if $\Ad_t(l) = l$ for any $t \in T$.
	
  Clearly there is a $T$-action on a projective variety $\IGr(d, \g_{-1})$. Since $T$ is
  solvable, by the Borel fixed point theorem~\cite[Theorem 6.2.6]{s} there is a fixed point $l$.
  Then $\n_l$ is $T$-invariant and the existence of the desired subgroup $N_l$ follows.
\end{proof}

Choose $l \in \IGr(d, \g_{-1})$ from the proof of Lemma \ref{constr Nl} (we will need a maximal torus
preserving $\n_l$ afterwards). If $\g_{-1} = 0$, then there is no ambiguity in the choice of $l$
and $\n_l = \m_l = \g_{\mr -2}$. Clearly $N_l$ acts by conjugation on $\g$, and we denote the
corresponding map $\ba \colon N_l \times \g \to \g$. Denote by $\m_l^\bot$ the orthogonal
complement to $\m_l$ in $\g$ with respect to the Killing form.

\begin{lem} \label{incl}
	$\ba(N_l \times \cS_e) \sb e + \m_l^\bot$.
\end{lem}

\begin{proof}
  Let us show that $\ker \ad_f \sb \m_l^\bot$, i.e.~$(y, z) = 0$ for any $y \in \ker \ad_f$
  and $z \in \m_l$. It follows from the inclusions $\ker \ad_f \sb \g_{\mr 0}$,
  $\m_l \sb \g_{\mr -1}$. Next, since $N_l$ is connected and we have the exponential map
  $\n_l\iso N_l$, it is enough to show that for any $x \in \n_l$ and any $y \in \ker \ad_f$
  we have $[x, y + e] \sb \m_l^\bot$, i.e.~$([x, y + e], z) = 0$ for any $z \in \m_l$. First
  let us consider the case $y = 0$. Since the Killing form is invariant, we have
  $([x, e], z) = (x, [e, z])$. Since $x \in \g_{\mr -1}$, $[e, z] \in \g_{\mr 1}$, the non-zero
  pairing can occur only for $x, z \in \g_{-1}$. But in this case by construction $x \in l$,
  $z \in l^{\bot_\om}$, so we always get zero. Second, let us prove $([x, y], z) = 0$ for any $x \in \n_l$, $y \in \m_l$, $z \in \ker \ad_f$. In this case by the representation theory of $\sl_2(\C)$ over $\C$ we have $z \in \g_{\mr 0}$. Since $x, y \in \g_{\mr -1}$, we have $[x, y] \in \g_{\mr -2}$, and hence $([x, y], z) = 0$.
\end{proof}

One of the main results of this section is the following proposition (cf.~\cite[Lemma~2.1]{gg}).

\begin{prop} \label{Ginz-Pr}
	The action map $\ba \colon N_l \times \cS_e \to e + \m_l^\bot$ is an isomorphism.
\end{prop}

\begin{proof}
  Since $N_l \times \cS_e$ is irreducible and $e + \m_l^\bot$ is normal, by Zariski main
  theorem~\cite[Theorem 5.2.8]{s} it is enough to prove that the morphism
  $\ba \colon N_l \times \cS_e \to e + \m_l^\bot$ is bijective and birational. Since
  $\char \C = 0$, by~\cite[Theorem 5.1.6]{s}, in order to show birationality it is enough to prove that
  the number of points in a general fiber equals one. So it is enough to show that
  $\ba \colon N_l \times \cS_e \to e + \m_l^\bot$ is bijective. Thus it is enough to work in the
  category of Lie groups. \newline
	
	Let $\ga \colon \SL_2(\C) \to G$ be the homomorphism of Lie groups corresponding to the
 	homomorphism of Lie algebras $\sl_2 \inj \g$, given by the $\sl_2$-triple $(e, f, h)$.
 	Such $\ga$ exists by \cite[Th.~3.41]{k}. Set
	\[\t{\ga} \colon \C^\times \to G,\ \t{\ga}(t) = \ga \begin{pmatrix} t & 0\\ 0 & t^{-1} \end{pmatrix} \]
	We suppose $h = \text{diag}(1, -1)$, so the differential of $\t{\ga}$ at the neutral element $\se$ of $G$ maps $\f{\p}{\p t} \in T_1 \C^\times$ to $h$.
	Clearly $\Ad_{\ga(t)}(e) = t^2 \cd e$, $\Ad_{\ga(t)}(f) = t^{-2} \cd f$. We consider the following $\C^\times$-action on $\g$: $\rho(t)(x) = t^2 \cd \Ad_{\ga(t^{-1})})(x)$ for any $x \in \g$, $t \in \C^\times$. Then $\rho(t)(e) = e, \rho(t)(f) = t^4 \cd f$ for any $t \in \C^\times$, so $\rho(t)$ stabilizes $\ker \ad_f$ and hence $\cS_e$. Since $\ad_h$ preserves $\m_l^{\bot_\om}$, by a standard property of the exponential map (see \cite[\S 3.1]{k}) $\rho(t)$ preserves $\m_l^{\bot_\om}$. So $\rho(t)$ preserves $e + \m_l^{\bot_\om}$. Finally, let us introduce a $\C^\times$-action on $N_l \times \cS_e$ as follows: $t \cd (g, x) = (\ga(t^{-1}) g \ga(t), \rho(t)(x))$. One can easily see that $t \cd (\se, e) = (\se, e)$ and $\ba \colon N_l \times \cS_e \to e + \m_l^\bot$ is $\C^\times$-equivariant map.
	
	\begin{lem} \label{contr}
	  Both $\C^\times$-actions on $N_l \times \cS_e$ and on $e + \m_l^{\bot}$ are contracting with fixed points $(\se, e)$ and $e$ respectively.
	\end{lem}
	
	\begin{proof}
	  First let us prove that $\C^\times$-actions on $e + \m_l^\bot$ and $\cS_e$ are
          contracting with fixed points equal to $e$. Note that if $x \in \g_i$, then
          $\rho(t)(x) = t^{2 - i} \cd x$ for any $t \in \C^\times$. As follows from the
          representation theory of $\sl_2$, $\ker \ad_f \sb \g_{\mr 0}$, so $\C^\times$-action
          on $\cS_e$ is contracting with fixed point $e$. A similar argument can be applied to
          $e + \m_l^\bot$, if we note $\m_l^\bot \sb \g_{\mr 1}$, since $\g_{\mr -2} \sb \m_l$.
		
	  Second let us show that the $\C^\times$-action on $N_l \times \cS_e$ is contracting.
          Actually we need to show that $\C^\times$-action on $N_l$ is contracting. Here
          the explicit construction of $N_l$ given in the proof of Proposition \ref{lie exist} turns
          very useful. Let $T$ be a maximal torus such that its adjoint action preserves $\n_l$,
          and let $h \in \Lie T$. Let $R_{\n_l} = \{\al_1, \ld, \al_h\}$ (see the notation
          of Proposition \ref{lie exist}). So it is enough to check that $t \cd u_{\al_j}(x)$ tends to~$0$
          for any $x \in \C$ and $1 \mr j \mr h$. But $t$ acts on $\Im(u_{\al_j}) = U_{\al_j}$ by
          multiplication with $\al_j(\rho(t)) = t^{-\al_j(h)}$, where $\al_j(h) \br 1$, and that implies the desired claim.
	\end{proof}
	
	So we have a $\C^\times$-equivariant morphism $\ba \colon N_l \times \cS_e \to e + \m_l^\bot$, and both sides are contractible with fixed points $(\se, e)$ and $e = \ba(\se, e)$. Our next goal is to show that $d_{(\se, e)} \ba$ is surjective.
	
	\begin{lem} \label{ml formula}
		$\m_l^\bot = [\n_l, e] \oplus \ker \ad_f$.
	\end{lem}
	
	\begin{proof}
	  By the $\sl_2$-representation theory, we have
          $\dim \ker \ad_f = \dim \g_0 + \dim \g_1$. So
          $\dim \n_l + \dim \g_0 + \dim \g_1 = \dim \g_{\mr -2} + \dim l + \dim \g_1 + \dim \g_0$, $\dim \m_l^\bot = \dim \g - \dim \m_l = \dim \g - \dim \g_{\mr -2} - \dim l^{\bot_\om} = \dim \g_{\mr -2} + 2\dim \g_1 + \dim \g_0 - (\dim \g_1 - \dim l) = \dim \g_{\mr -2} + \dim \g_1 + \dim \g_0 + \dim l$. So $\dim \m_l^\bot = \dim [\n_l, e] + \dim \ker \ad_f$. Next,  the
          $\sl_2$-representation theory implies $[\n_l, e] \cap \ker \ad_f = 0$. So it remains
          to show that $[\n_l, e], \ker \ad_f \sb \m_l^\bot$. But it is already proved in Lemma \ref{incl}.
	\end{proof}
	
	As follows from Lemma \ref{ml formula}, $\dim (N_l \times \cS_e) = \dim (e + \m_l^\bot)$ and
        $d_{(\se, e)} \ba$ is surjective. Therefore, $d_{(\se, e)} \ba$
        is bijective and by the Implicit Function Theorem $\ba$ is an isomorphism of some open neighborhood of $(\se, e)$ (in the classical topology) onto some open neighborhood of $e$ in $e + \m_l^\bot$. Lemma \ref{contr} implies that $\ba$ is bijective, so we are done.
\end{proof}

Recall that we have a character $\chi \in \n_l^\ast$ such that $\chi(x) = (e, x)$. Set
$\psi = \chi |_{N_l}$. Let us prove that $\psi$ is $N_l$-invariant. Since $N_l$ is connected, it is enough to prove that $\psi \circ \ad_y = 0$ for any $y \in \n_l = \Lie N_l$. Indeed, for any $x \in \n_l$ one can write $\psi([x, y]) = ([x, y], e) = (x, [y, e]) = 0$, since $x \in \g_{\mr -1}$, $[y, e] \in \g_{\mr 1}$, and if $x, y \in \g_{\mr -1}$, then $x, y \in l$, and so $\psi([x, y]) = 0$, because $l$ is isotropic. As we know, the twisted cotangent bundle $T_\psi^\ast(G/N_l)$ is an affine bundle over $G/N_l$, and the fiber over $[g]$ can be identified with
$\{\b \in T_g^\ast G : \b(u_{G, x}^R) = \psi(x) \text{ for any $x \in \n_l$}\}$.
The following proposition is the main result of this section.

\begin{prop} \label{slice = twisted}
	Suppose that $l$ is a Lagrangian subspace. Then there is a natural $G$-equivariant isomorphism 
	\[\Xi \colon T_\psi^\ast(G/N_l) \iso G \times \cS_e\]
\end{prop}

\begin{proof}
  Since $l$ is Lagrangian, $\n_l = \m_l$. Recall that there is an isomorphism
  $\Phi \colon \g \iso \g^\ast$ given by the Killing form. Then $\Phi(\n_l^\bot) = \Ann(\n_l)$,
  $\Phi(e) = \chi$, so $\Phi(e + \n_l^\bot) = \chi + \Ann(\n_l)$ "coincides" with the fiber of
  $T_\psi^\ast(G/N_l)$ over $[\se]$. This allows us to contruct $\Xi$ explicitly. Recall that we
  constructed $T_\psi^\ast(G/N_l)$ in two steps. First we considered
  $Y = \mu_{N_l}^{-1}(\psi) \sb T^\ast G$ and then took the quotient by the right $N_l$-action.
  Consider the projection $\pi \colon Y \to T_\psi^\ast(G/N_l)$. We will construct $\Xi$ also in
  two steps: first we construct a morphism $\t{\Xi} \colon Y \to G \times \cS_e$, and then
  we check that it descends to the desired isomorphism.
	
  Denote the left $g$-translation on $G$ by $L_g$, so that $L_g(h) = gh$. Then the left
  $g$-action on $T^\ast G$ is $(h, \b) \mapsto (gh, L_{g^{-1}}^\ast(\b))$. Note that
  $L_g^\ast(\b) \in \Phi(e + \n_l^\bot)$, so $\ba^{-1}(\Phi^{-1}(\b))$ is well-defined and equals
  $(h, x)$ for some $x \in \cS_e$, $h \in N_l$. Set $\t{\Xi}(g, \a) = (gh, x)$,
  i.e.~$\t{\Xi}(g, \a) = g \cd \ba^{-1}(\Phi^{-1}(L_g^\ast(\a)))$, where $g$ acts by the left
  multiplication on the first coordinate. Clearly we get a regular surjective $G$-equivariant
  morphism, where $G$ acts on the left.
	
  Let us show that $\t{\Xi}$ is $N_l$-equivariant, where $N_l$ acts on the right. Since
  $\t{\Xi}$ is $G$-equivariant, it is enough to prove that for any $\b \in \Phi(e + \n_l^\bot)$
  (i.e.~over $\se$ in $Y$) and any $q \in \t{\Xi}$ we have $\Ad_q^\ast(\b) = \b$ (since we
  first act by the right translation by $q^{-1}$ and then by the left translation by $q$ to
  return a point in the fiber over $\se$). The latter equality follows from $N_l$-invariance of
  $\psi$. So $\t{\Xi}$ descends to a well-defined surjective $G$-equivariant morphism $\Xi \colon T_\psi^\ast(G/N_l) \sur G \times \cS_e$. Clearly by Proposition \ref{Ginz-Pr} $\Xi$ is bijective, $G \times \cS_e$ is normal, $T_\psi^\ast(G/N_l)$ is irreducible, so by \cite[Theorem 5.1.6]{s} $\Xi$ is birational, and by Zariski main theorem~\cite[Theorem 5.2.8]{s} it is an isomorphism.
\end{proof}

\begin{rem} \label{Xi section}
  Note that the map $\t{\Xi} \colon Y \to G \times \cS_e$ constructed in the proof of
  the previous theorem has a section $s$
  such that $s(g, x) = (d_\se L_{g^{-1}})^\ast(\Phi(x)) \in T_g^\ast G$.
  Indeed, $\t{\Xi}$ is $G$-equivariant (with respect to the left action), $(\se, \Phi(x)) \in Y$, so $(d_\se L_{g^{-1}})^\ast(\Phi(x)) \in Y$. To find $\t{\Xi}(d_\se L_{g^{-1}})^\ast(\Phi(x))$ one needs first to
  move the point to the fiber of $Y$ over $\se$ and then apply $\ba^{-1} \circ \Phi^{-1}$. But after moving we get $\Phi(x)$, where already $x \in \cS_e$, so $\t{\Xi}(s(g, x)) = (g, x)$. 
\end{rem}

Assume that a closed connected subgroup $S \sb G$ stabilizes $\cS_e$, i.e.~$p\cS_e p^{-1} = \cS_e$ for any $p \in S$. Then there is a natural $S$-action on $G \times \cS_e$: $p \cd (g, x) = (gp^{-1}, \Ad_p(x))$. Our goal is to relate this action to the one in \S\ref{tcb 2sided}. One of the assumptions in \S\ref{tcb 2sided} was $S \sb N_G(N_l)$ ($N_l$ here corresponds to $H$ there). However, in case $G = G_2$ such condition does not hold for any $S$ and $l$, so we suppose now that $S \sb N_G(N_l)$. Also assume that $\psi |_{[\fs, \n_l]} = 0$.

\begin{lem} \label{id 2sided}
  The $S$-action on $T_\psi^\ast(G/N_l)$ coming from the right $S$-action on $T^\ast G$
  and defined in \S\ref{tcb 2sided}, coincides via the natural isomorphism
  of Proposition~\ref{slice = twisted} with the natural $S$-action on $G \times \cS_e$ defined in the previous paragraph.
\end{lem}

\begin{proof}
  Note that both actions commute with the left $G$-action on $T_\psi^\ast(G/N_l)$ and
  $G \times \cS_e$, so it is enough to consider only points in the fiber over $[\se] \in G/N_l$. Suppose that $y = (\se, \ga) \in Y = \mu_{N_l}^{-1}(\psi)$. Take any $p \in S$. By explicit construction of $S$-action on $T_\psi^\ast(G/N_l)$ we have $p \cd \pi(\se, \ga) = \pi(p^{-1}, \b)$, where $\pi \colon Y \to T_\psi^\ast(G/N_l)$ is the natural projection and $\ga = (d_\se R_{p^{-1}})^\ast(\b)$ ($R_p$ is the right translation by $p^{-1}$, and $\b \in T_p^\ast G$). Now let us apply $\Xi = \t{\Xi} \circ \pi$ from the proof of Proposition \ref{slice = twisted} to $(p^{-1}, \b)$.
  First, we need to return this point to the fiber of $Y$ over $\se$ by the left translation by
  $p$. Hence we get $(\se, \Ad_{p^{-1}}^\ast(\ga))$. Then we need to find
  $(h', x') \in N_l \times \cS_e$ such that $\Phi(\ba(h', x')) = \Ad_{p^{-1}}^\ast(\ga)$, where
  $\Phi \colon \g \iso \g^\ast$ is given by the Killing form, and set $\Xi(p^{-1}, \b) = (p^{-1} h', x')$. For this, one may choose $(h, x)$ so that the equality $\Phi(\ba(h, x)) = \ga$ holds true. All we need to show is that $p \cd (h, x) = (ph', x')$, i.e.~$hp^{-1} = p^{-1}h'$ and $\Ad_p(x) = x'$. 
	
  By definition of $\ba$ we have $\Phi(\Ad_h(x)) = \ga$, $\Phi(\Ad_{h'}(x')) = \Ad_{p^{-1}}^\ast(\ga)$ and these equalities uniquely determine $(h, x)$ and $(h', x')$. The latter equality can be rewritten as $(\Ad_{h'}(x'), w) = \ga(\Ad_{p^{-1}}(w))$ for any $w \in \g$. Since the Killing form is $G$-invariant, the latter is equivalent to $(\Ad_{p^{-1} h'}(x'), w) = \ga(w)$ for any
  $w \in \g$, i.e.~$\Phi(\Ad_{p^{-1}h'}(x')) = \ga$. So we have $\Ad_h(x) = \Ad_{p^{-1}h'}(x')$ and
  this equality uniquely determines $h', x'$ if we already know $h, x$. Take $x'' = \Ad_p(x)$, $h'' = \Ad_{p}(h)$. Since $p$ stabilizes $N_l$ and $\cS_e$, we have $x'' \in \cS_e$ and
  $h'' \in N_l$. All in all we have $\Ad_{p^{-1} h''}(x'') = \Ad_{p^{-1} h'' p}(x) = \Ad_{\Ad_{p^{-1}}(h'')}(x) = \Ad_h(x)$, so $h'' = h'$ and $x'' = x'$.
\end{proof}


\subsubsection{Losev symplectic structure} \label{id losev}

As in \ref{id ginz}, let $G$ be a connected semisimple linear algebraic group over $\C$. Recall that there is a natural isomorphism $\Phi \colon \g \iso \g^\ast$ given by the Killing form on $\g$. 

\begin{lem} \label{id_los 1}
  The morphism $\Psi \colon G \times \g \to T^\ast G$, $(g, x) \mapsto (d_\se L_{g^{-1}})^\ast(\Phi(x)) \in T_g^\ast G$, where $L_g$ is the left translation by $g$ on $G$, is an isomorphism of algebraic varieties.
\end{lem} 

\begin{proof}
  Clearly $\Psi$ is bijective (so it is birational, since $\char(\C) = 0$,
  see~\cite[Theorem 5.1.6]{s}), $T^\ast G$ is normal, $G \times \g$ is irreducible, so by Zariski Main
  Theorem~\cite[Theorem 5.2.8]{s} $\Psi$ is an isomorphism.
\end{proof}

Pulling back by $\Psi$ we get a symplectic form $\om'$ on $G \times \g$. There is a natural
$G \times G$-action on $G \times \g$: $(g, s) \cdot (p, x) = (gps^{-1}, \Ad_s(x))$. Recall
that $u_{G, x}^L$ denotes the velocity field with respect to $x \in \g$ of the left $G$-action on $G$. Similarly for $u_{G, x}^R$, $u_{T^\ast G, x}^L$ and $u_{T^\ast G, x}^R$.

\begin{lem} \label{id_los 2}
	The natural $G \times G$-action on $G \times \g$ is Hamiltonian and corresponds to the natural two-sided $G\times G$-action on $T^\ast G$. Moreover, the Hamiltonian for $(x, y) \in \g \oplus \g$ equals $H_{(x, y)}(g, z) = (z, (d_\se L_g)^{-1}(u_{G, x}^R - u_{G, y}^L))$.
\end{lem}

\begin{proof}	
  First we show that the left $G$-action on $G \times \g$ corresponds to the left
  $G$-action on $T^\ast G$. Indeed, let $(p, x)$ correspond to $\ga \in T_p^\ast G$, that is
  $\Phi(x) = (d_\se L_p)^\ast(\ga)$. Next we apply the left translation by $g \in G$ and get $\b \in T_{gp}^\ast G$ such that~$(d_pL_g)^\ast(\b) = \ga$. Since $L_{gp} = L_g L_p$, we have $(d_\se L_{gp})^\ast = (d_\se L_p)^\ast (d_\se L_g)^\ast$. This implies that $\b \in T_{gp}^\ast G$ corresponds to $(gp, x) \in G \times \cS_e$ via $\Psi$.
	
  Next let us analyze the following $G$-action on $G \times \g$:
  $s \cdot (g, x) = (gs^{-1}, \Ad_s(x))$, in terms of $T^\ast G$. We claim that it corresponds to
  the natural right $G$-action on $T^\ast G$. Indeed, let $(p, x)$ correspond to
  $\ga \in T_p^\ast G$, that is $\Phi(x) = (d_\se L_p)^\ast(\ga)$. Next we apply multiplication by $g^{-1} \in G$ on the right and get $\b \in T_{pg^{-1}}^\ast G$ such that~$\b = (d_{pg^{-1}} R_g)^\ast(\ga)$. Let us check that $(pg^{-1}, \Ad_g(x))$ corresponds to $\b$, so we need to prove $(d_\se L_{pg^{-1}})^\ast(\b) = \Phi(\Ad_g(x))$, i.e.~$(d_\se L_{pg^{-1}})^\ast(\b))(y) = (\Ad_g(x), y)$ for any $y \in \g$. Note that 
	\[(d_\se L_{pg^{-1}})^\ast(\b)(y) = \b(d_\se L_{pg^{-1}}(y)) = (d_{pg^{-1}} R_g \circ d_\se L_{pg^{-1}})(y) = (d_\se L_p)(\Ad_{g^{-1}}(y)) = \ga(\Ad_{g^{-1}}(y)) = (x, \Ad_{g^{-1}}(y))\] 
	Since the Killing form is invariant, $(x, \Ad_{g^{-1}}(y)) = (\Ad_g(x), y)$ and we are done. The explicit formula for Hamiltonian follows from Fact \ref{tcb 1} and the similar fact for the left action discussed in the beginning of \S\ref{tcb la}.
\end{proof}

The tangent space to $G \times \g$ at the point $(g, x)$ can be identified with $\g \oplus \g$ as follows: $(u, v) \in \g \oplus \g$ corresponds to $(d_\se L_g(u), v) \in T_g G \oplus \g = T_{(g, x)}(G \times \g)$. Recall that $\om'$ is a symplectic form on $G \times \g$ induced from $T^\ast G$.

\begin{lem} [{cf.~\cite[\S3.1, $4^{th}$~par.]{l2}}] \label{id_los 3}
	For any $w_1 = (u_1, v_1), w_2 = (u_2, v_2) \in \g \oplus \g \simeq T_{(g, x)}(G \times \g)$ holds $\om'(w_1, w_2) = -(x, [u_1, u_2]) - (u_1, v_2) + (u_2, v_1)$, where $(\cd,\:\cd)$ stands for the Killing form on $\g$.
\end{lem}

\begin{proof}
  First let us show that $\om'$ has the desired form on $T_{(\se, x)} (G \times \g)$. Clearly it is enough to consider only three cases: (i) $w_1 = (u, 0)$, $w_2 = (0, v)$; (ii) $w_{1, 2} = (0, v_{1, 2})$; (iii) $w_{1, 2} = (u_{1, 2}, 0)$. Recall that the value of symplectic form $\om$ on $T^\ast G$ on two tangent vectors $(\xi_{1, 2}, \nu_{1, 2}) \in T_{(g, \a)}(T^\ast G) = T_g G \oplus T_g^\ast G$ equals $\nu_1(\xi_2) - \nu_2(\xi_1)$ (see \S\ref{tcb constr}). Note that $d_{\se, x}\Psi$ maps any vector $v$ tangent to $\g$ to a vector
  $(0, \Phi(v))$ tangent to a fiber of $T^\ast G$. Also $d_{\se, x}\Psi$ maps any tangent vector of the form $(u, 0) \in T_{\se, x} (G \times \g)$ to a tangent vector of the form $(u, \xi) \in T_{\se, \Phi(x)} (T^\ast G)$ (it is important that the first projection is preserved). So in case (i) we get two vectors $(u, \xi)$ and $(0, \Phi(v))$ in $T_{\se, \Phi(x)} (T^\ast G)$, so the value of $\om$ on them is $-\Phi(v)(u) = -(u, v)$. In case (ii) we get two vectors $(0, \Phi(v_{1, 2}))$ in $T_{\se, \Phi(x)} (T^\ast G)$, so the value of $\om$ on them is zero.
	
	The most interesting case is (iii), because we do not know $\xi$ (see the previous paragraph) explicitly. By definition of $\Psi$ the value of $\om'$ on $(u_1, 0)$ and $(u_2, 0)$ equals $\om(u_{T^\ast G, u_1}^L, u_{T^\ast G, u_2}^L)|_{(\se, x)}$. The left $G$-action on $T^\ast G$ is Hamiltonian with Hamiltonian $x \mapsto H_x = -\s_1(u_{G, x}^L)$. It means that $\om(\cd, u_{T^\ast G, x}^L) = dH_x$ for any $x \in \g$. So by \cite[formula (1.2.3)]{cg} we have $\om(u_{T^\ast G, y}^L, u_{T^\ast G, z}^L) = \{H_y, H_z\}$, where $\{\cd, \: \cd\}$ stands for the Poisson bracket on $\O(T^\ast G)$. By definition of the Hamiltonian action (see \cite[Definition 1.4.1, \S 1.2]{cg}) we have $\{H_y, H_z\} = H_{[y, z]}$ for any $y, z \in \g$. So $\om(u_{T^\ast G, u_1}^L, u_{T^\ast G, u_2}^L)|_{(\se, \Phi(x))} = H_{[u_1, u_2]}(\se, \Phi(x)) = -\Phi(x)([u_1, u_2]) = -(x, [u_1, u_2])$ and we are done for the points of the form $(\se, x)$.
	
	Now it remains to show that $\om'$ has the desired form at any point in $G \times \g$. We actually need to compute $\om$ on two vectors $d_{(g, x)} \Psi((d_\se L_g)(u_{1, 2}), v_{1, 2})$.
        Note that $\om$ is $G$-equivariant $2$-form with respect to the action coming from the
        left $G$-action on $G$. The left action by $g^{-1}$ on $T^\ast G$ corresponds to the following action on $G \times \g$: $(p, y) \mapsto (g^{-1}p, y)$. So it is enough to compute $\om'$ on the pair $(d_g R_{g^{-1}}(d_\se R_g)(u_{1, 2}), v_{1, 2}) = (u_{1, 2}, v_{1, 2})$, and we have done
        this already.
\end{proof}

Since $\scrX = G \times \cS_e \sb G \times \g$, we can restrict the symplectic form $\om'$ to
$\scrX$. This construction appears in~\cite[\S3.1, $4^{th}$~paragraph]{l2}. The tangent space to $\scrX$ at the point $(g, x)$ can be identified with $\g \oplus \ker \ad_f$ as follows: $(u, v)$ corresponds to $(d_\se L_g(u), v) \in T_g G \oplus T_x \cS_e$. The key fact is that the resulting $2$-form on $\scrX$ is symplectic. Although is was proved in~\cite[\S2]{l1} in a large generality, we give a shorter proof in our particular case.

\begin{lem} \label{id_los 4}
	The $2$-form $\om'|_{\scrX}$ is symplectic.
\end{lem}

\begin{proof}
  Since the pull-back of a closed $2$-form is again closed, $\om'|_{\scrX}$ is closed.
  The only thing we need to show is that $\om'$ is non-degenerate. Take any non-zero
  $(u, v) \in T_{(g, x)} \scrX \simeq \g \oplus \ker \ad_f$. If there exists $w \in \g$
  such that $-(x, [u, w]) + (w, v) \ne 0$, then $\om'((u, v), (w, 0)) \ne 0$ and we are done.
  So suppose $-(x, [u, \cd]) + (v, \cd) \equiv 0$, that is $[x, u] = v$, since the Killing form is invariant and non-degenerate. If there exists $w \in \ker \ad_f$ such that $(u, w) \ne 0$, then $\om'((u, v), (0, w)) \ne 0$ and we are done. So suppose $u \perp \ker \ad_f$. For $y \in \g$, by $y_i$ we denote the $h$-weight component of $y$ lying in $\g_i$. Consider the maximal $i$ such that $u_i \ne 0$. Let us show that $[e, u_i] \ne 0$. It is clear that $u_i \bot \ker \ad_f$, since $\ker \ad_f$ also has the $h$-weight decomposition. From the $\sl_2$-representation theory we know that $\ker \ad_f \oplus \Im(\ad_e) = \g$, so that $u_i \not \bot \Im(\ad_e)$, and hence $(u_i, [w, e]) \ne 0$ for some $w \in \g$. Therefore, $([u_i, e], w) = (u_i, [e, w]) \ne 0$, so that $[u_i, e] \ne 0$. Since $x \in e + \ker \ad_f$ and $\ker \ad_f \sb \g_{\mr 0}$, the maximal $j$ such that $[x, u]_j \ne 0$ is equal to $i + 2$, and in this case $[x, u]_{i+2} = [e, u_i] \ne 0$. Since $[e, u_i] \ne 0$, it follows from the $\sl_2$-representation theory that $[f, [e, u_i]] \ne 0$. Meanwhile, it is clear that $[f, [x, u]]_i = [f, [x, u]_{i+2}] = [f, [e, u_i]] \ne 0$, what contradicts $[x, u] \in \ker \ad_f$.
\end{proof}

Suppose that $S \sb G$ is a connected closed subgroup that stabilizes $\cS_e$. Then
a natural $G \times S$-action on $G \times \cS_e$ (namely, $(g, s) \cdot (p, x) = (gps^{-1}, \Ad_s(x))$)
is Hamiltonian, since the ambient $G \times G$-action (see Lemma \ref{id_los 2} and a paragraph before) is Hamiltonian.

Recall that in \S\ref{id ginz} we have constructed an isomorphism $\Xi \colon T_{\psi}^\ast(G/N_l) \iso G \times \cS_e$. We have a symplectic structure on $T_\psi^\ast(G/N_l)$ discussed in \S\ref{tcb constr}, that induces a symplectic structure on $G \times \cS_e$.

\begin{prop} \label{id_los 5}
	The symplectic structures on $G \times \cS_e$ coincide: one coming from the isomorphism $G \times \cS_e \simeq T_\psi^\ast(G/N_l)$ (see Proposition \ref{slice = twisted}), and another coming from $G \times \g$.
\end{prop}

\begin{proof}
  Note that both symplectic forms are invariant with respect to the left $G$-action, so it is enough to consider only points of the form $(\se, x)$. Let $Y = \mu_{N_l}^{-1}(\psi)$, so
  $T_\psi^\ast(G/N_l) = Y \sslash N_l$ and $\pi \colon Y \to T_\psi^\ast(G/N_l)$ is the quotient map. Let $s \colon G \times \cS_e \to Y$ be a section of $\t{\Xi}$ (see Remark \ref{Xi section}).
  Recall that $s(g, x) = (d_\se L_{g^{-1}})^\ast(\Phi(x))$, so this section can be extended to the
  whole $G \times \g$ (of course, $Y$ is replaced by $T^\ast G$). But then we get $\Psi$
  defined in Lemma \ref{id_los 1}, so $s = \Psi |_\scrX$. Recall that the symplectic form on $T_\psi^\ast(G/N_l)$ was induced from $Y$ (and hence from $T^\ast G$), while the symplectic form on $G \times \cS_e$ was induced from $G \times \g$. Furthermore, the symplectic forms on $G \times \g$ and
  $T^\ast G$ match under the isomorphism $\Psi$. It follows that the symplectic forms on
  $G \times \cS_e$ and $T_\psi^\ast(G/N_l)$ match under the isomorphism $\Xi$, and we are done.
\end{proof}
 

\subsubsection{Counting relevant orbits} \label{id ro}

Recall what we are proving. Consider the simple linear algebraic group $G$ of type $G_2$ over
$\C$ and fix a maximal torus with Cartan Lie algebra $\ft$. We will use the explicit
description of $\g_2$ given in \S\ref{ap g2}. Let $e = e_1$ be a nilpotent element corresponding
to the short root $\al$ in~Figure~\ref{1}.

\begin{figure}[h]
\[\begin{tikzpicture}[x=0.75pt,y=0.75pt,yscale=-1,xscale=1]
	
	\draw [color={rgb, 255:red, 0; green, 0; blue, 0 }  ,draw opacity=1 ]   (270.08,160.32) -- (337.66,199.12) ;
	\draw [shift={(339.4,200.12)}, rotate = 209.86] [color={rgb, 255:red, 0; green, 0; blue, 0 }  ,draw opacity=1 ][line width=0.75]    (10.93,-3.29) .. controls (6.95,-1.4) and (3.31,-0.3) .. (0,0) .. controls (3.31,0.3) and (6.95,1.4) .. (10.93,3.29)   ;
	\draw [color={rgb, 255:red, 0; green, 0; blue, 0 }  ,draw opacity=1 ]   (270.08,160.32) -- (337.47,121.19) ;
	\draw [shift={(339.2,120.19)}, rotate = 149.86] [color={rgb, 255:red, 0; green, 0; blue, 0 }  ,draw opacity=1 ][line width=0.75]    (10.93,-3.29) .. controls (6.95,-1.4) and (3.31,-0.3) .. (0,0) .. controls (3.31,0.3) and (6.95,1.4) .. (10.93,3.29)   ;
	\draw [color={rgb, 255:red, 0; green, 0; blue, 0 }  ,draw opacity=1 ] [dash pattern={on 4.5pt off 4.5pt}]  (270.08,160.32) -- (269.89,82.4) ;
	\draw [shift={(269.88,80.4)}, rotate = 89.86] [color={rgb, 255:red, 0; green, 0; blue, 0 }  ,draw opacity=1 ][line width=0.75]    (10.93,-3.29) .. controls (6.95,-1.4) and (3.31,-0.3) .. (0,0) .. controls (3.31,0.3) and (6.95,1.4) .. (10.93,3.29)   ;
	\draw [color={rgb, 255:red, 0; green, 0; blue, 0 }  ,draw opacity=1 ]   (270.08,160.32) -- (270.27,238.25) ;
	\draw [shift={(270.28,240.25)}, rotate = 269.86] [color={rgb, 255:red, 0; green, 0; blue, 0 }  ,draw opacity=1 ][line width=0.75]    (10.93,-3.29) .. controls (6.95,-1.4) and (3.31,-0.3) .. (0,0) .. controls (3.31,0.3) and (6.95,1.4) .. (10.93,3.29)   ;
	\draw [color={rgb, 255:red, 0; green, 0; blue, 0 }  ,draw opacity=1 ] [dash pattern={on 0.84pt off 2.51pt}]  (270.08,160.32) -- (202.69,199.45) ;
	\draw [shift={(200.96,200.45)}, rotate = 329.86] [color={rgb, 255:red, 0; green, 0; blue, 0 }  ,draw opacity=1 ][line width=0.75]    (10.93,-3.29) .. controls (6.95,-1.4) and (3.31,-0.3) .. (0,0) .. controls (3.31,0.3) and (6.95,1.4) .. (10.93,3.29)   ;
	\draw [color={rgb, 255:red, 0; green, 0; blue, 0 }  ,draw opacity=1 ] [dash pattern={on 0.84pt off 2.51pt}]  (202.5,121.52) -- (270.08,160.32) ;
	\draw [shift={(200.76,120.53)}, rotate = 29.86] [color={rgb, 255:red, 0; green, 0; blue, 0 }  ,draw opacity=1 ][line width=0.75]    (10.93,-3.29) .. controls (6.95,-1.4) and (3.31,-0.3) .. (0,0) .. controls (3.31,0.3) and (6.95,1.4) .. (10.93,3.29)   ;
	\draw [color={rgb, 255:red, 0; green, 0; blue, 0 }  ,draw opacity=1 ] [dash pattern={on 0.84pt off 2.51pt}]  (247.89,122.29) -- (270.08,160.32) ;
	\draw [shift={(246.88,120.56)}, rotate = 59.74] [color={rgb, 255:red, 0; green, 0; blue, 0 }  ,draw opacity=1 ][line width=0.75]    (10.93,-3.29) .. controls (6.95,-1.4) and (3.31,-0.3) .. (0,0) .. controls (3.31,0.3) and (6.95,1.4) .. (10.93,3.29)   ;
	\draw [color={rgb, 255:red, 0; green, 0; blue, 0 }  ,draw opacity=1 ]   (270.08,160.32) -- (292.27,122.29) ;
	\draw [shift={(293.28,120.56)}, rotate = 120.26] [color={rgb, 255:red, 0; green, 0; blue, 0 }  ,draw opacity=1 ][line width=0.75]    (10.93,-3.29) .. controls (6.95,-1.4) and (3.31,-0.3) .. (0,0) .. controls (3.31,0.3) and (6.95,1.4) .. (10.93,3.29)   ;
	\draw [color={rgb, 255:red, 0; green, 0; blue, 0 }  ,draw opacity=1 ]   (270.08,160.32) -- (314.08,160.17) ;
	\draw [shift={(316.08,160.16)}, rotate = 179.8] [color={rgb, 255:red, 0; green, 0; blue, 0 }  ,draw opacity=1 ][line width=0.75]    (10.93,-3.29) .. controls (6.95,-1.4) and (3.31,-0.3) .. (0,0) .. controls (3.31,0.3) and (6.95,1.4) .. (10.93,3.29)   ;
	\draw [color={rgb, 255:red, 0; green, 0; blue, 0 }  ,draw opacity=1 ] [dash pattern={on 0.84pt off 2.51pt}]  (270.08,160.32) -- (225.28,160.55) ;
	\draw [shift={(223.28,160.56)}, rotate = 359.71] [color={rgb, 255:red, 0; green, 0; blue, 0 }  ,draw opacity=1 ][line width=0.75]    (10.93,-3.29) .. controls (6.95,-1.4) and (3.31,-0.3) .. (0,0) .. controls (3.31,0.3) and (6.95,1.4) .. (10.93,3.29)   ;
	\draw [color={rgb, 255:red, 0; green, 0; blue, 0 }  ,draw opacity=1 ]   (270.08,160.32) -- (292.28,198.83) ;
	\draw [shift={(293.28,200.56)}, rotate = 240.03] [color={rgb, 255:red, 0; green, 0; blue, 0 }  ,draw opacity=1 ][line width=0.75]    (10.93,-3.29) .. controls (6.95,-1.4) and (3.31,-0.3) .. (0,0) .. controls (3.31,0.3) and (6.95,1.4) .. (10.93,3.29)   ;
	\draw [color={rgb, 255:red, 0; green, 0; blue, 0 }  ,draw opacity=1 ]   (270.08,160.32) -- (247.89,198.04) ;
	\draw [shift={(246.88,199.76)}, rotate = 300.47] [color={rgb, 255:red, 0; green, 0; blue, 0 }  ,draw opacity=1 ][line width=0.75]    (10.93,-3.29) .. controls (6.95,-1.4) and (3.31,-0.3) .. (0,0) .. controls (3.31,0.3) and (6.95,1.4) .. (10.93,3.29)   ;
	\draw [color={rgb, 255:red, 0; green, 0; blue, 0 }  ,draw opacity=1 ]   (195.12,229.36) -- (347.28,90.72) ;
	
	\draw (328.32,152.16) node [anchor=north west][inner sep=0.75pt]  [font=\scriptsize]  {$\alpha $};
	\draw (186.32,110.36) node [anchor=north west][inner sep=0.75pt]  [font=\scriptsize]  {$\beta $};
\end{tikzpicture}\]
\caption{Roots of $G_2$}
\label{1}
\end{figure}
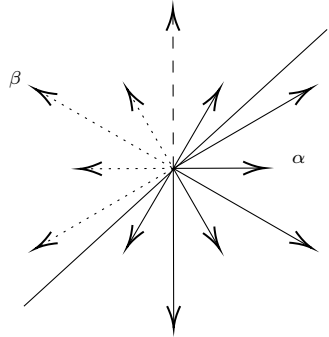

Clearly there is a nilpotent element $f = -f_1$, corresponding to the opposite root $-\al$, and $h = 3 E_{11} - \Id \in \ft$ such that $(e, f, h)$ form an $\sl_2$-triple. Set $\cS_e = e + \ker \ad_f$  and put $\scrX = G \times \cS_e$. Recall the symplectic structure on $\scrX$ defined in~\S\ref{id losev}. There are two roots orthogonal to $\pm\al$, namely $\pm (3\al + 2\be)$. Consider the corresponding semi-simple subgroup $R$ isomorphic to $\SL_2$. Indeed, $R$ is a semisimple group of rank $1$, so $R \cong \SL_2$ or $R \cong \PSL_2$. Since $R$ has an irreducible 2-dimensional subrepresentation in $\g_2$ formed by a direct sum of root subspaces corresponding to $\al + \be$, $-\be - 2\al$, and $\PSL_2$ has no irreducible 2-dimensional representations, we get $R \not \cong \PSL_2$. 

The subgroup $R$ acts on $\cS_e$ by conjugation,
and there is a Hamiltonian (see Lemma \ref{id_los 2}) $G \times R$-action on
$\scrX$: $(g, p) \cd (q, x) = (gqp^{-1}, \Ad_p(x))$, where $g, q \in G$, $x \in \cS_e$,
$p \in R$. Set $\t{G} = G \times R$ and consider a Borel subgroup $\t{B} \sb \t{G}$.
Then the $\t{B}$-action on $\scrX$ is also Hamiltonian, so one can consider the
\emph{moment map} $\mu_{\t{B}} \colon \scrX \to (\Lie \t{B})^\ast$. Set
$\La_\scrX = \mu_{\t{B}}^{-1}(0)$. Then we are going to prove that $\La_\scrX$ is Lagrangian and
has~$7$ irreducible components.

Consider the $h$-weight decomposition discussed in \S\ref{id ginz}. Since the the present part of the paper by $\g_2$ we denote the Lie algebra of $G_2$, we write the weights with respect to $h$ in brackets.
\[\begin{tabular}{ | m{2cm} | m{2cm}| m{2cm} | m{2cm} | m{2cm} | m{2cm} | m{2cm} |} 
	\hline
	$\g_{(-3)}$ & $\g_{(-2)}$ & $\g_{(-1)}$ & $\g_{(0)}$ & $\g_{(1)}$ & $\g_{(2)}$ & $\g_{(3)}$\\ 
	\hline
	$E_{21}, E_{31}$ & $f_1$ & $e_2, e_3$ & $\ft, E_{23}, E_{32}$ & $f_2, f_3$ & $e_1$ & $E_{12}, E_{13}$\\ 
	\hline
\end{tabular}\]
Since $\dim \g_{(-1)} = 2$, any one-dimensional subspace is Lagrangian. Choose
$l = \Span_\C(e_2)$. It is $T$-invariant, and $\n_l = \Span_\C(E_{21}, E_{31}, f_1, e_2)$
(it corresponds to the weights $-3\al - \be, -\al, \be, \al + \be$, see the dotted weights in~Figure~\ref{1}). Let $N_l$ be a closed connected subgroup of $G$ with the Lie algebra $\n_l$. It was constructed in~\S\ref{ap sub}. In Proposition \ref{slice = twisted} we have constructed a $G$-equivariant isomorphism $\scrX \iso T_{\psi}^\ast(G/N_l)$. 

Since all Borel subgroups are conjugate, the choice of $\t{B}$ plays no role. Denote by $S$
the Borel subgroup of $R$ with Lie algebra $\fs = \Span_\C(E_{23}, E_{22} - E_{33})$
(its unipotent part corresponds to the weight $3\al + 2\be$, see the dashed weight
in~Figure~\ref{1}). Then it is clear from~Figure~\ref{1} that $[\fs, \n_l] \sb \n_l$. Let $B$ be the Borel subgroup in $G$ containing $T$ and the roots to the
north-east of the line in Figure \ref{1}.
Then one can choose $\t{B} = B \times S$. By construction in \S\ref{tcb 2sided} there is a natural two-sided action
$\t{B} \circar T_\psi(G/N_l)$ coming from the two-sided action on $T^\ast G$. By Lemma \ref{id_los 2}
this $\t{B}$-action matches the $\t{B}$-action on $\scrX$ under the isomorphism
$\scrX \simeq T_{\psi}^\ast(G/N_l)$. Also by Proposition \ref{id_los 5} the symplectic structures on
$G \times \cS_e$ (defined in \S\ref{id losev}) and on $T_\psi^\ast(G/N_l)$ (defined in \S\ref{tcb constr}) coincide. Thus we can analyze two-sided $\t{B}$-action on $T_\psi^\ast(G/N_l)$.
By Proposition \ref{tcb 16}, $\La_\scrX$ is Lagrangian and the number of its irreducible components equals
the number of relevant $B$-orbits in $G/S \ltimes N_l$. So it only remains to count the
relevant orbits. \newline

Denote by $T' \sb T$ the maximal torus in $S$ with Lie algebra $\ft' = \Span_\C(E_{22} - E_{23})$.
Set $U_4 = N_l$, $U_6 = B_u$, i.e.~the unipotent subgroup with Lie algebra
$\u_6 = \Span_\C(E_{31}, f_1, E_{21}, e_2, E_{23}, f_3)$ (and weights
$-3\al - \be, -\al, \be, \al + \be, 3\al + 2\be, 2\al + \be$ in~Figure~\ref{1}). Also one can easily
prove, using results of \S\ref{ap sub}, that $U_5 \rtimes T' = S \ltimes N_l$, where $U_5$ is the
unipotent subgroup with Lie algebra $\u_5 = \Span_\C(E_{31}, f_1, E_{21}, e_2, E_{23})$
(and weights $-3\al - \be, -\al, \be, \al + \be, 3\al + 2\be$ in~Figure~\ref{1}). So
$U_4 \sb U_5 \sb U_6$. We want to count the relevant $B$-orbits in $G/U_5 \rtimes T'$.
By~Definition~\ref{tcb 10} it is equivalent to counting $U_5 \rtimes T'$-orbits $\bO$ in $G/B$
such that $\psi |_{\Lie\Stab_{U_5 \rtimes T'}(x)} = 0$ for some (and hence any) $x \in \bO$. 

Let $W = N_G(T)/T$ be the Weyl group of $(G,T)$. For $w \in W$ denote by $\dot w$ a
representative of $w$ in $N_G(T)$. Note that $\dot w B$ is $T$-fixed point in $\B$.
Consider the Bruhat decomposition of $\B = G/B$ (see \cite[\S8.3]{s}), so
\[\B = \bigsqcup_{w \in W} B \dot w B = \bigsqcup_{w \in W} U_6 \dot w B = \bigsqcup_{w \in W} U_{w^{-1}} \dot w B\]
where $U_{w^{-1}}$ is a unipotent subgroup of $U_6$ with Lie algebra $\bigoplus_{\al \in R^+(B) \cap w(- R^+(B))} \g_\al$ (see \cite[\S 8.3.1, Lemma 8.3.6]{s}). Each \emph{Bruhat cell} $C(w) = U_{w^{-1}} \dot w B$ is isomorphic to
$U_{w^{-1}}$ via $u \mapsto u {\dot w} B$. Clearly $U_5$ (and even $U_6$) preserves any Bruhat cell while acting on the left. Also each Bruhat cell is a single $U_6$-orbit (where $U_6$ acts, as usual, on the left). For $w \in W$ denote by $\o{U}_w$ the $T$-stable unipotent subgroup of $U_5$ such that $\o{\u}_w = \Lie \o{U}_w = \bigoplus_{\al \in R(U_5) \cap w(R^+(B))} \g_\al$, where $R(U_5) = \{-3\al - \be, -\al, \be, \al + \be, 3\al + 2\be\}$, so that $\u_5 = \bigoplus_{\al \in R(U_5)} \g_\al$. 

\begin{lem} \label{ro 1}
	For any $w \in W$ the corresponding Bruhat cell $C(w)$ is either a single $U_5 \rtimes T'$-orbit or a disjoint union of two $U_5 \rtimes T'$-orbits. Furthermore, $\Lie \Stab_{U_5}(\dot w B) = \o{\u}_w$.
\end{lem}

\begin{proof}
	Since $u \dot w B = \dot w \cd w^{-1}(u)B$, $u \in \Stab_{U_5}(\dot w B)$ iff $w^{-1}(u) \in B$, what clearly implies $\Lie \Stab_{U_5}(\dot w B) = \o{\u}_w$.
	
	Next, consider the Bruhat cell $C(w) = U_{w^{-1}} \dot w B$. As follows from the proof of \cite[Lemma 8.3.6]{s} and the structure theory of $T$-invariant unipotent subgroups in $U_6 = B_u$, for any $T$-stable connected unipotent subgroup $U \sb U_6$ the $U$-orbits in $C(w)$ coincide with $(U \cap U_{w^{-1}})$-orbits in $C(w)$. Set $\u_5' = \u_5 \cap \Lie U_{w^{-1}}$, and let $U_5'$ be the corresponding closed connected unipotent subgroup in $G$ (it exists by \S\ref{ap sub}). Then $U_5' = U_5 \cap  U_{w^{-1}}$ and the $U_5$-orbits in $C(w)$ coincide with $U_5'$-orbits in $C(w)$. There are two possible cases.
	
	\textsc{Case 1}: $U_5' = U_{w^{-1}}$. In this case $C(w)$ is a single $U_5$-orbit.
	
	\textsc{Case 2}: $U_5' \sbn U_{w^{-1}}$. Then the $U_5' \rtimes T'$-orbit through the
        point $\dot w$ cannot be the whole $C(w)$, because $\dot w$ is a $T'$-fixed point and
        $C(w)$ is not a single $U_5'$-orbit. Since $\u_5$ and $\u_6$ differ only in one weight subspace
        $\g_\ga$ ($\ga = 2\al + \be$), $\Lie U_{w^{-1}} = \u_5' \oplus \g_\ga$, so
        $U_{w^{-1}} \simeq U_\ga \ltimes U_5'$. Here the product is semi-direct, because
        $(U_5, U_\ga) \sb U_5$ and $U_5' \sb U_5$, so $U_5'$ is a normal subgroup in $U_{w^{-1}}$. Thus
        any subset $\{u v \dot w B : u \in U_\ga, v \in U_5\}$ is a single $U_5'$-orbit in $C(w)$,
        and we get an infinite set of $U_5'$-orbits indexed by $U_\ga$: the $U_5'$-orbit corresponding
        to $u \in U_\ga$ equals $\{u v \dot w B : v \in U_5'\}$. The torus $T'$ acts on the set of
        $U_5'$-orbits in $C(w)$, so acts on $U_\ga$ algebraically. Since $U_\ga$ has a non-zero
        $T'$-weight, $T'$ cannot stabilize any of the $U_5'$-orbits $\{u v \dot w B : v \in U_5'\}$ for $u \ne 1$. It means that the $T'$-action on $U_\ga \simeq \bA^1$ has only one fixed point (the neutral
        element), so other points lie in a single orbit. Therefore, the complement of $U_5' \dot w B$ in $C(w)$ is a single $U_5 \rtimes T'$-orbit.
\end{proof}

Note that $\o{\u}_w$ is $T$-invariant, hence the $U_5 \rtimes T'$-orbit through $\dot w B$ is
relevant iff $f_1 \notin \o{\u}_w$. This can be seen in~Figure~\ref{1}. 
It is well known (see \cite[Corollary 6.4.12]{s}) that the set of all Borel subgroups in $G$ containing $T$ is
in bijection with $W$ via the following map: $w \in W$ corresponds to $\dot w B \dot w^{-1}$.
Furthermore, $R^+ (\dot w B \dot w^{-1}) = w(R^+(B))$. Also it is well known that Borel subgroups in
$G$ containing $T$ are in one-to-one correspondence with positive root systems in
$R(G)$~\cite[\S8.2]{s}.

We are interested in all $\al \in R(U_5) \cap w(R^+(B)) = R(U_5) \cap R^+(\dot w B \dot w^{-1})$, then the direct sum of the corresponding weight subspaces will be equal to $\o{\u}_w$. So the weights of $\o{\u}_w$ are exactly
all the dotted/dashed weights in Figure \ref{1} lying in $w(R^+(B))$. Since $\psi |_\ft = \psi |_{\ft'} = 0$, it is enough to
check $\psi |_{\o{\u}_w} = 0$ in order to guarantee that the $U_5 \rtimes T'$-orbit through
$\dot w B$ is relevant. Thus we need to find all $w \in W$ such that $(-\al) \notin w(R^+(B))$, because $(-\al)$ is the only dotted/dashed weight such that $\psi$ does not vanish on the corresponding weight subspace. Since $w(R^+(B))$ runs through the set of polarizations of $R(B_2)$ as $w$ runs through $W$, we need to find all polarizations not containing $(-\al)$. There are~$6$
such polarizations. An easy way to obtain this number is to notice that for any two opposite
polarizations exactly one of them contains $(-\al)$, and the number of all polarizations equals $12$.

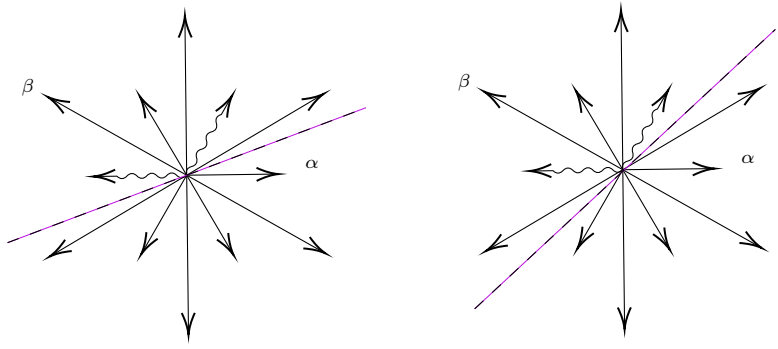
\begin{figure}[h]
	\[\begin{tikzpicture}[x=0.75pt,y=0.75pt,yscale=-1,xscale=1]
		
		\draw [color={rgb, 255:red, 0; green, 0; blue, 0 }  ,draw opacity=1 ]   (410.05,142.49) -- (477.97,180.69) ;
		\draw [shift={(479.71,181.67)}, rotate = 209.35] [color={rgb, 255:red, 0; green, 0; blue, 0 }  ,draw opacity=1 ][line width=0.75]    (10.93,-3.29) .. controls (6.95,-1.4) and (3.31,-0.3) .. (0,0) .. controls (3.31,0.3) and (6.95,1.4) .. (10.93,3.29)   ;
		\draw [color={rgb, 255:red, 0; green, 0; blue, 0 }  ,draw opacity=1 ]   (410.05,142.49) -- (477.09,102.76) ;
		\draw [shift={(478.81,101.74)}, rotate = 149.35] [color={rgb, 255:red, 0; green, 0; blue, 0 }  ,draw opacity=1 ][line width=0.75]    (10.93,-3.29) .. controls (6.95,-1.4) and (3.31,-0.3) .. (0,0) .. controls (3.31,0.3) and (6.95,1.4) .. (10.93,3.29)   ;
		\draw [color={rgb, 255:red, 0; green, 0; blue, 0 }  ,draw opacity=1 ]   (410.05,142.49) -- (409.16,64.57) ;
		\draw [shift={(409.14,62.57)}, rotate = 89.35] [color={rgb, 255:red, 0; green, 0; blue, 0 }  ,draw opacity=1 ][line width=0.75]    (10.93,-3.29) .. controls (6.95,-1.4) and (3.31,-0.3) .. (0,0) .. controls (3.31,0.3) and (6.95,1.4) .. (10.93,3.29)   ;
		\draw [color={rgb, 255:red, 0; green, 0; blue, 0 }  ,draw opacity=1 ]   (410.05,142.49) -- (410.93,220.41) ;
		\draw [shift={(410.95,222.41)}, rotate = 269.35] [color={rgb, 255:red, 0; green, 0; blue, 0 }  ,draw opacity=1 ][line width=0.75]    (10.93,-3.29) .. controls (6.95,-1.4) and (3.31,-0.3) .. (0,0) .. controls (3.31,0.3) and (6.95,1.4) .. (10.93,3.29)   ;
		\draw [color={rgb, 255:red, 0; green, 0; blue, 0 }  ,draw opacity=1 ]   (410.05,142.49) -- (343.01,182.22) ;
		\draw [shift={(341.29,183.24)}, rotate = 329.35] [color={rgb, 255:red, 0; green, 0; blue, 0 }  ,draw opacity=1 ][line width=0.75]    (10.93,-3.29) .. controls (6.95,-1.4) and (3.31,-0.3) .. (0,0) .. controls (3.31,0.3) and (6.95,1.4) .. (10.93,3.29)   ;
		\draw [color={rgb, 255:red, 0; green, 0; blue, 0 }  ,draw opacity=1 ]   (342.12,104.3) -- (410.05,142.49) ;
		\draw [shift={(340.38,103.31)}, rotate = 29.35] [color={rgb, 255:red, 0; green, 0; blue, 0 }  ,draw opacity=1 ][line width=0.75]    (10.93,-3.29) .. controls (6.95,-1.4) and (3.31,-0.3) .. (0,0) .. controls (3.31,0.3) and (6.95,1.4) .. (10.93,3.29)   ;
		\draw [color={rgb, 255:red, 0; green, 0; blue, 0 }  ,draw opacity=1 ]   (387.52,104.65) -- (410.05,142.49) ;
		\draw [shift={(386.49,102.94)}, rotate = 59.23] [color={rgb, 255:red, 0; green, 0; blue, 0 }  ,draw opacity=1 ][line width=0.75]    (10.93,-3.29) .. controls (6.95,-1.4) and (3.31,-0.3) .. (0,0) .. controls (3.31,0.3) and (6.95,1.4) .. (10.93,3.29)   ;
		\draw [color={rgb, 255:red, 0; green, 0; blue, 0 }  ,draw opacity=1 ]   (410.05,142.49) .. controls (409.43,140.22) and (410.26,138.77) .. (412.53,138.15) .. controls (414.8,137.53) and (415.63,136.08) .. (415.01,133.81) .. controls (414.39,131.54) and (415.22,130.09) .. (417.49,129.47) .. controls (419.76,128.85) and (420.59,127.4) .. (419.97,125.13) .. controls (419.35,122.86) and (420.18,121.41) .. (422.45,120.79) .. controls (424.72,120.17) and (425.55,118.72) .. (424.93,116.45) .. controls (424.32,114.17) and (425.15,112.72) .. (427.42,112.1) -- (427.93,111.21) -- (431.9,104.26) ;
		\draw [shift={(432.89,102.52)}, rotate = 119.75] [color={rgb, 255:red, 0; green, 0; blue, 0 }  ,draw opacity=1 ][line width=0.75]    (10.93,-3.29) .. controls (6.95,-1.4) and (3.31,-0.3) .. (0,0) .. controls (3.31,0.3) and (6.95,1.4) .. (10.93,3.29)   ;
		\draw [color={rgb, 255:red, 0; green, 0; blue, 0 }  ,draw opacity=1 ]   (410.05,142.49) -- (454.04,141.94) ;
		\draw [shift={(456.04,141.92)}, rotate = 179.29] [color={rgb, 255:red, 0; green, 0; blue, 0 }  ,draw opacity=1 ][line width=0.75]    (10.93,-3.29) .. controls (6.95,-1.4) and (3.31,-0.3) .. (0,0) .. controls (3.31,0.3) and (6.95,1.4) .. (10.93,3.29)   ;
		\draw [color={rgb, 255:red, 0; green, 0; blue, 0 }  ,draw opacity=1 ]   (410.05,142.49) .. controls (408.4,144.18) and (406.74,144.2) .. (405.05,142.56) .. controls (403.36,140.92) and (401.7,140.94) .. (400.05,142.63) .. controls (398.4,144.32) and (396.74,144.34) .. (395.05,142.7) .. controls (393.36,141.06) and (391.7,141.08) .. (390.05,142.77) .. controls (388.4,144.46) and (386.74,144.48) .. (385.05,142.84) .. controls (383.36,141.2) and (381.7,141.22) .. (380.05,142.91) .. controls (378.4,144.6) and (376.74,144.62) .. (375.05,142.98) -- (373.25,143) -- (365.25,143.12) ;
		\draw [shift={(363.25,143.14)}, rotate = 359.2] [color={rgb, 255:red, 0; green, 0; blue, 0 }  ,draw opacity=1 ][line width=0.75]    (10.93,-3.29) .. controls (6.95,-1.4) and (3.31,-0.3) .. (0,0) .. controls (3.31,0.3) and (6.95,1.4) .. (10.93,3.29)   ;
		\draw [color={rgb, 255:red, 0; green, 0; blue, 0 }  ,draw opacity=1 ]   (410.05,142.49) -- (432.59,180.8) ;
		\draw [shift={(433.6,182.52)}, rotate = 239.52] [color={rgb, 255:red, 0; green, 0; blue, 0 }  ,draw opacity=1 ][line width=0.75]    (10.93,-3.29) .. controls (6.95,-1.4) and (3.31,-0.3) .. (0,0) .. controls (3.31,0.3) and (6.95,1.4) .. (10.93,3.29)   ;
		\draw [color={rgb, 255:red, 0; green, 0; blue, 0 }  ,draw opacity=1 ]   (410.05,142.49) -- (388.2,180.4) ;
		\draw [shift={(387.2,182.13)}, rotate = 299.96] [color={rgb, 255:red, 0; green, 0; blue, 0 }  ,draw opacity=1 ][line width=0.75]    (10.93,-3.29) .. controls (6.95,-1.4) and (3.31,-0.3) .. (0,0) .. controls (3.31,0.3) and (6.95,1.4) .. (10.93,3.29)   ;
		\draw [color={rgb, 255:red, 0; green, 0; blue, 0 }  ,draw opacity=1 ][fill={rgb, 255:red, 189; green, 16; blue, 224 }  ,fill opacity=1 ] [dash pattern={on 4.5pt off 4.5pt}]  (335.7,212.19) -- (486.62,72.21) ;
		\draw [color={rgb, 255:red, 0; green, 0; blue, 0 }  ,draw opacity=1 ]   (191.18,145.24) -- (259.11,183.43) ;
		\draw [shift={(260.85,184.41)}, rotate = 209.35] [color={rgb, 255:red, 0; green, 0; blue, 0 }  ,draw opacity=1 ][line width=0.75]    (10.93,-3.29) .. controls (6.95,-1.4) and (3.31,-0.3) .. (0,0) .. controls (3.31,0.3) and (6.95,1.4) .. (10.93,3.29)   ;
		\draw [color={rgb, 255:red, 0; green, 0; blue, 0 }  ,draw opacity=1 ]   (191.18,145.24) -- (258.22,105.51) ;
		\draw [shift={(259.94,104.49)}, rotate = 149.35] [color={rgb, 255:red, 0; green, 0; blue, 0 }  ,draw opacity=1 ][line width=0.75]    (10.93,-3.29) .. controls (6.95,-1.4) and (3.31,-0.3) .. (0,0) .. controls (3.31,0.3) and (6.95,1.4) .. (10.93,3.29)   ;
		\draw [color={rgb, 255:red, 0; green, 0; blue, 0 }  ,draw opacity=1 ]   (191.18,145.24) -- (190.3,67.32) ;
		\draw [shift={(190.28,65.32)}, rotate = 89.35] [color={rgb, 255:red, 0; green, 0; blue, 0 }  ,draw opacity=1 ][line width=0.75]    (10.93,-3.29) .. controls (6.95,-1.4) and (3.31,-0.3) .. (0,0) .. controls (3.31,0.3) and (6.95,1.4) .. (10.93,3.29)   ;
		\draw [color={rgb, 255:red, 0; green, 0; blue, 0 }  ,draw opacity=1 ]   (191.18,145.24) -- (192.07,223.16) ;
		\draw [shift={(192.09,225.16)}, rotate = 269.35] [color={rgb, 255:red, 0; green, 0; blue, 0 }  ,draw opacity=1 ][line width=0.75]    (10.93,-3.29) .. controls (6.95,-1.4) and (3.31,-0.3) .. (0,0) .. controls (3.31,0.3) and (6.95,1.4) .. (10.93,3.29)   ;
		\draw [color={rgb, 255:red, 0; green, 0; blue, 0 }  ,draw opacity=1 ]   (191.18,145.24) -- (124.14,184.97) ;
		\draw [shift={(122.42,185.98)}, rotate = 329.35] [color={rgb, 255:red, 0; green, 0; blue, 0 }  ,draw opacity=1 ][line width=0.75]    (10.93,-3.29) .. controls (6.95,-1.4) and (3.31,-0.3) .. (0,0) .. controls (3.31,0.3) and (6.95,1.4) .. (10.93,3.29)   ;
		\draw [color={rgb, 255:red, 0; green, 0; blue, 0 }  ,draw opacity=1 ]   (123.26,107.04) -- (191.18,145.24) ;
		\draw [shift={(121.51,106.06)}, rotate = 29.35] [color={rgb, 255:red, 0; green, 0; blue, 0 }  ,draw opacity=1 ][line width=0.75]    (10.93,-3.29) .. controls (6.95,-1.4) and (3.31,-0.3) .. (0,0) .. controls (3.31,0.3) and (6.95,1.4) .. (10.93,3.29)   ;
		\draw [color={rgb, 255:red, 0; green, 0; blue, 0 }  ,draw opacity=1 ]   (168.65,107.4) -- (191.18,145.24) ;
		\draw [shift={(167.63,105.68)}, rotate = 59.23] [color={rgb, 255:red, 0; green, 0; blue, 0 }  ,draw opacity=1 ][line width=0.75]    (10.93,-3.29) .. controls (6.95,-1.4) and (3.31,-0.3) .. (0,0) .. controls (3.31,0.3) and (6.95,1.4) .. (10.93,3.29)   ;
		\draw [color={rgb, 255:red, 0; green, 0; blue, 0 }  ,draw opacity=1 ]   (191.18,145.24) .. controls (190.56,142.97) and (191.39,141.52) .. (193.66,140.9) .. controls (195.93,140.28) and (196.76,138.83) .. (196.15,136.56) .. controls (195.53,134.29) and (196.36,132.84) .. (198.63,132.22) .. controls (200.9,131.6) and (201.73,130.15) .. (201.11,127.88) .. controls (200.49,125.61) and (201.32,124.16) .. (203.59,123.53) .. controls (205.86,122.91) and (206.69,121.46) .. (206.07,119.19) .. controls (205.45,116.92) and (206.28,115.47) .. (208.55,114.85) -- (209.07,113.95) -- (213.04,107.01) ;
		\draw [shift={(214.03,105.27)}, rotate = 119.75] [color={rgb, 255:red, 0; green, 0; blue, 0 }  ,draw opacity=1 ][line width=0.75]    (10.93,-3.29) .. controls (6.95,-1.4) and (3.31,-0.3) .. (0,0) .. controls (3.31,0.3) and (6.95,1.4) .. (10.93,3.29)   ;
		\draw [color={rgb, 255:red, 0; green, 0; blue, 0 }  ,draw opacity=1 ]   (191.18,145.24) -- (235.18,144.69) ;
		\draw [shift={(237.18,144.67)}, rotate = 179.29] [color={rgb, 255:red, 0; green, 0; blue, 0 }  ,draw opacity=1 ][line width=0.75]    (10.93,-3.29) .. controls (6.95,-1.4) and (3.31,-0.3) .. (0,0) .. controls (3.31,0.3) and (6.95,1.4) .. (10.93,3.29)   ;
		\draw [color={rgb, 255:red, 0; green, 0; blue, 0 }  ,draw opacity=1 ]   (191.18,145.24) .. controls (189.54,146.93) and (187.87,146.95) .. (186.18,145.31) .. controls (184.49,143.67) and (182.83,143.69) .. (181.18,145.38) .. controls (179.53,147.07) and (177.87,147.09) .. (176.18,145.45) .. controls (174.49,143.81) and (172.83,143.83) .. (171.18,145.52) .. controls (169.54,147.21) and (167.88,147.23) .. (166.19,145.59) .. controls (164.5,143.95) and (162.84,143.97) .. (161.19,145.66) .. controls (159.54,147.35) and (157.88,147.37) .. (156.19,145.73) -- (154.39,145.75) -- (146.39,145.86) ;
		\draw [shift={(144.39,145.89)}, rotate = 359.2] [color={rgb, 255:red, 0; green, 0; blue, 0 }  ,draw opacity=1 ][line width=0.75]    (10.93,-3.29) .. controls (6.95,-1.4) and (3.31,-0.3) .. (0,0) .. controls (3.31,0.3) and (6.95,1.4) .. (10.93,3.29)   ;
		\draw [color={rgb, 255:red, 0; green, 0; blue, 0 }  ,draw opacity=1 ]   (191.18,145.24) -- (213.73,183.54) ;
		\draw [shift={(214.74,185.27)}, rotate = 239.52] [color={rgb, 255:red, 0; green, 0; blue, 0 }  ,draw opacity=1 ][line width=0.75]    (10.93,-3.29) .. controls (6.95,-1.4) and (3.31,-0.3) .. (0,0) .. controls (3.31,0.3) and (6.95,1.4) .. (10.93,3.29)   ;
		\draw [color={rgb, 255:red, 0; green, 0; blue, 0 }  ,draw opacity=1 ]   (191.18,145.24) -- (169.33,183.15) ;
		\draw [shift={(168.33,184.88)}, rotate = 299.96] [color={rgb, 255:red, 0; green, 0; blue, 0 }  ,draw opacity=1 ][line width=0.75]    (10.93,-3.29) .. controls (6.95,-1.4) and (3.31,-0.3) .. (0,0) .. controls (3.31,0.3) and (6.95,1.4) .. (10.93,3.29)   ;
		\draw [color={rgb, 255:red, 0; green, 0; blue, 0 }  ,draw opacity=1 ][fill={rgb, 255:red, 189; green, 16; blue, 224 }  ,fill opacity=1 ] [dash pattern={on 4.5pt off 4.5pt}]  (101.28,179.04) -- (281.09,111.44) ;
		
		\draw (468.21,133.81) node [anchor=north west][inner sep=0.75pt]  [font=\scriptsize,rotate=-359.49]  {$\alpha $};
		\draw (325.85,93.28) node [anchor=north west][inner sep=0.75pt]  [font=\scriptsize,rotate=-359.49]  {$\beta $};
		\draw (249.35,136.56) node [anchor=north west][inner sep=0.75pt]  [font=\scriptsize,rotate=-359.49]  {$\alpha $};
		\draw (106.98,96.02) node [anchor=north west][inner sep=0.75pt]  [font=\scriptsize,rotate=-359.49]  {$\beta $};

	\end{tikzpicture}\]
	
	\caption{Possible polarizations for $(-\al), \ga \notin R^+(\dot w B \dot w^{-1})$}
	\label{2}
\end{figure}

Finally, let us find all relevant orbits among the complements to $U_5 \dot w B$ in $C(w)$, where
$w \in W$. First we need the existence of such complement. It is equivalent to $U_\ga \sb U_{w^{-1}}$
($\ga = 2\al + \be$), i.e. $\ga \in w(-R^+(B)) \Leftrightarrow \ga \notin w(R^+(B))$. Also as in the previous paragraph, it is enough
to check $\psi |_{\Lie \Stab_{U_5}(u \dot w B)} = \psi |_{\Ad_u(\o{\u}_w)} = 0$ for some (and hence for any)
$u \in U_\ga$, $u \ne 1$. Considering the limit $u \to 1$ we claim that if the complement is relevant,
then the $U_5 \rtimes T'$-orbit through $\dot w B$ is also relevant. Hence we may consider
only the complements to relevant orbits, which means that $(-\al)$ is not in the system of positive roots corresponding to
$\dot w B \dot w^{-1}$.

So we get two possible polarizations (wavy weights $(-\al)$ and $\ga$ should not lie in the system of positive roots, see Figure \ref{2}) given by the lower half-planes with respect to the dashed lines. For the left polarization $\o{\u}_w = \g_{-3\al - \be}$, for the right one $\o{\u}_w = 0$. Thus the right
polarization precisely gives a relevant orbit. For the left polarization we need to check
$\psi|_{\Ad_u(\g_{-3\al - \be})} = 0$ for any $u \in U_\ga$. If this is so, then taking the differential of
$\Ad_u$ at $u = 1$, we get $\psi|_{[x, y]} = 0$ for any $x \in \g_\ga, y \in \g_{-3\al-\be}$. But
$[\g_\ga, \g_{-3\al-\be}] = \g_{\ga + (-3\al-\be)} = \g_{-\al}$ and $\psi$ does not vanish on it. So the
"complementary" orbit corresponding to the left polarization is not relevant. Thus we get only one relevant "complementary" orbit.

Therefore, we get $7$ relevant orbits in total and complete the proof of the first part of Theorem \ref{theo}. 


\section{Linear side} \label{ls}

\subsection{Relation with conormal bundles to relevant orbits} \label{ls rel}

By \cite[Theorem 1.5.7]{cg}, the variety $\La_{\scrX^\vee}$ is coisotropic. We view $\t{G}^\vee$ as a subgroup of $\Sp_{14} \sb \GL_{14} = \GL(\scrX^\vee)$. The symplectic form $\om$ on $\C^7 \ot \C^2$ is given by a tensor product of a $G_2$-invariant non-degenerate symmetric bilinear form $\<\cd, \cd\>$ on $\C^7$ (see Lemma~\ref{ap 7rep 3}) and an $\SL_2$-invariant symplectic form $\om_2$ on $\C^2$.

Let $l \sb \C^2$ be a line. Choose a Borel subgroup $\t{B}^\vee$ of the form $B \times B_2$, where
$B_2 \sb \SL_2$ is a Borel subgroup preserving $l$, and $B \sb G_2$ is a Borel subgroup.
We can decompose $\C^2 = l \oplus l'$, $l = \Span_\C(e_1)$, $l' = \Span_\C(e_2)$, $\om_2(e_1, e_2) = 1$,
and take $B_2$ equal to the subgroup of upper-triangular matrices in the base $(e_1, e_2)$. Set
$L = \C^7 \ot l$, $L' = \C^7 \ot l'$ and $V = \scrX^\vee/L$. There is a natural $\t{B}^\vee$-action on
$V$. For any $x \in \scrX^\vee$ denote by $[x]$ its image under the natural projection
$\scrX^\vee \to V$. It is easy to see from explicit formula for $\om$ that $L$ is Lagrangian subspace
in $\scrX^\vee$. So $\om$ gives rise to a natural isomorphism $\Xi \colon V^\ast \iso \C^7 \ot L$,
i.e.~$x \in L^\ast$ corresponds to $[y] \mapsto \om(x, y)$. Note that a (co)tangent space to $V$ at
any point can be identified with $V$ (resp.~$V^\ast \simeq L$), so that $T^\ast V = V \oplus L$.

\begin{lem} \label{ls rel 1}
  The correspondence $\Phi \colon (x', x) \mapsto ([x'], x)$ is a $B$-equivariant isomorphism
  of symplectic varieties $\Phi \colon \scrX^\vee = L' \oplus L \iso T^\ast V = V \oplus L$,
  where the $B$-action on $T^\ast V$ comes from the $B$-action on $V$.
\end{lem}

\begin{proof}
  Clearly $\Phi$ is an isomorphism in the category of vector spaces. The symplectic form $\Om$
  on $T^\ast V$ equals $\Om(([y_1], x_1), ([y_2], x_2)) = \om(x_1, y_2) - \om(x_2, y_1)$, where
  $y_i \in \scrX^\vee$, $x_i \in L$. Let us show that $\Phi$ respects the symplectic structure. Take any
  two vectors $(y_1, x_1), (y_2, x_2) \in \scrX^\vee = L' \oplus L$. They map to
  $([y_1], x_1), ([y_2], x_2)$, so the value of $\Om$ on them equals $\om(x_1, y_2) - \om(x_2, y_1)$.
  To compute the value of symplectic form $\om$ on the pair $(y_1, x_1), (y_2, x_2)$, we write $y_i = y_i' \ot e_2$, $x_i = x_i' \ot e_1$, where $y_i', x_i' \in \C^7$. So $\om((y_1, x_1), (y_2, x_2)) = \om(y_1' \ot e_2 + x_1' \ot e_1, y_2' \ot e_2 + x_2' \ot e_1) = \<x_1', y_2'\> - \<x_2', y_1'\> = \om(x_1, y_2) - \om(x_2, y_1)$ by explicit formula for $\om$, given in the beginning of \S\ref{ls}.
	
  Finally, it remains to show that $\Phi$ is $B$-equivariant. Indeed, take any element
  $g \in B$. Let $g$ act by an operator $A$ on $\C^7$. Take any $(x', x) \in L' \oplus L = \scrX^\vee$.
  It corresponds to $([x'], x)$ under $\Phi$. If we first act by $(A, \Id)$ on $(x', x)$ and then
  apply $\Phi$, we get $([Ax'], Ax) \in T^\ast V = V \oplus L$. But $A$ acts on $V$ as follows:
  $[x] \mapsto [Ax]$. So the induced action on $T^\ast V$ is as follows:
  $([x], \a) \mapsto (A[x], (A^\ast)^{-1}(\a))$, where $\a \in T_{[x]}^\ast V \simeq V^\ast$. Since $A$
  preserves $\om$, the map $(A^\ast)^{-1}$ corresponds under $\Phi$ to the map $L \mapsto L$, $x \mapsto Ax$. So the $(A, \Id)$-action on $T^\ast V$ maps $([x'], x)$ to $([Ax'], Ax)$ and we are done.
\end{proof}

Since $V \simeq L'$ in a natural way, there is a $B$-invariant symmetric bilinear form on $V$,
defined as follows: $\<[x_1 \ot e_2], [x_2 \ot e_2]\> = \<x_1, x_2\>$. Set
$Q = \{x \in V : \<x, x\> = 0\}$. Note that $Q$ is $B$-invariant.

\begin{prop} \label{ls rel 2}
	The variety $\La_{\scrX^\vee}$ corresponds under $\Phi$ to the union of conormal bundles to $B$-orbits in $Q$.
\end{prop}

\begin{proof}
  Let us give a criterion when $(z' \ot e_2, z \ot e_1) \in L' \oplus L = \scrX^\vee$ lies in
  $\La_{\scrX^\vee}$. Recall from~\cite[Proposition 1.4.6]{cg} that the Hamiltonian of $\t{B}^\vee$-action on $\scrX^\vee$ is given by a formula $H_A(v) = \f{1}{2} \om(v, Av)$ up to a sign, where $A \in \Lie \t{B}^\vee \sb \sp_{14}$. Since $\Lie \t{B}^\vee = \Lie B \oplus \Lie B_2$, one can divide the condition $\mu_{\t{B}^\vee} = 0$ into two parts: for $A \in \Lie B$ and for $A \in \Lie B_2$. Let $v = z' \ot e_2 + z \ot e_1$. Then for $A \in \Lie B$ we get $\om(z' \ot e_2 + z \ot e_1, Az' \ot e_2 + Az \ot e_1) = 0$, that is $\<z, Az'\> - \<z', Az\> = 0$. For $A \in \Lie B_2$ we consider two cases: $A = E_{11} - E_{22}$ and $A = E_{12}$ (in the basis $(e_1, e_2)$). In the first case we get the condition $\om(z' \ot e_2 + z \ot e_1, -z' \ot e_2 + z \ot e_1) = -2\<z, z'\> = 0$. In the second case we get the condition $\om(z' \ot e_2 + z \ot e_1, z' \ot e_1) = -\<z', z'\> = 0$. So the full list of conditions is as follows:
	\begin{itemize}
		\item[(i)] $\<z', z'\> = 0$ (so that $[z' \ot e_2] \in Q$).
		\item[(ii)] $\<z, z'\> = 0$.
		\item[(iii)] $\<Az, z'\> = \<z, Az'\>$ for any $A \in \Lie B$.
	\end{itemize}
	The tangent space to the $B$-orbit through the point $[z' \ot e_2] \in V$ consists of vectors $[Az' \ot e_2]$, where $A \in \Lie B$. In order to show that the point gets into the conormal bundle to the orbit, we need to check that the pairing between $Az' \ot e_2$ and $z \ot e_1$ is zero, i.e.~$\<Az', z\> = 0$. Since the $B$-action on $\C^7$ preserves $\<\cd, \cd\>$, for any $A \in \Lie B$ we have $\<Az, z'\> + \<z, Az'\> = 0$, so $\<Az, z'\> = \<z, Az'\>$ implies $\<Az', z\> = 0$ for any $A \in \Lie B$. So the image of $\La_{\scrX^\vee}$ lies in the union of conormal bundles to $B$-orbits.
	
	It remains to prove that the conormal bundle of any $B$-orbit in $Q$ lies in the image of $\La_{\scrX^\vee}$. Clearly, conditions (i) and (iii) hold for any point in a conormal bundle to a $B$-orbit in $Q$. Moreover, in (iii) both sides vanish. It remains to prove $\<z, z'\> = 0$. Suppose $\<z, z'\> \ne 0$. Then $H_{A}(v) \ne 0$ for $A = E_{12} \in \Lie B_2$. Thus $H_{A'}(v) \ne 0$ for $A' = \e(E_{11} - E_{22}) + E_{12} \in \Lie B_2$ and $\e \ne 0$ small enough. However, $A'$ is $B_2$-conjugate to some $\la (E_{11} - E_{22})$, the moment map $B_2$-equivariant (and even $\t{B}$-equivariant), and by (iii) we already know that $H_{E_{11} - E_{22}}(v) = 0$, so $H_{A'}(v) = 0$. We get a contradiction. Therefore, $\<z, z'\> = 0$, and we get the whole conormal bundle.
\end{proof}

\begin{rem}
	(1) In both proofs of Lemma~\ref{ls rel 1} and Prop.~\ref{ls rel 2} we actually used neither that $V$ is seven-dimensional nor that $G = G_2$. Both proofs work without modifications for any finite-dimensional symplectic representation of any reductive connected linear algebraic group over $\C$.
	
	(2) The statement of Prop.~\ref{ls rel 2} is very similar to in \cite[Exercise 10.6]{l3} (also cf.~Remark~\ref{rem Los-rel-orb}). Indeed, if we consider only Hamiltonian $B$-action, then we will get exactly the statement of \cite[Exercise 10.6]{l3}. But considering $(B \times B_2)$-action instead of $B$-action we get additional conditions for the $B$-orbit in the base of conormal bundle (as one sees from the proof of Prop.~\ref{ls rel 2}). Namely, a relevant orbit should lie in the quadric $Q$.
\end{rem}

\begin{cor} \label{ls rel 3}
  If the number of $B$-orbits in $Q$ is finite, then $\La_{\scrX^\vee}$ is Lagrangian, and
  $\# \mathrm{Irr}\La_{\scrX^\vee}$ equals the number of $B$-orbits in $Q$.
\end{cor}

Indeed, $\La_{\scrX^\vee}$ is coisotropic, but if there are finitely many $B$-orbits in $Q$, it has
dimension $\f{1}{2} \dim \scrX^\vee$, so $\La_{\scrX^\vee}$ is Lagrangian. Each cotangent bundle to
a $B$-orbit is irreducible, so its closure contributes one irreducible component. It remains to count
the $B$-orbits in the quadric $Q$.


\subsection{Counting orbits in the quadric} \label{ls count}

Denote by $\o{Q}$ the projectivization $\P(Q)$ of $Q$ in $\P(\C^7)$. It contains the highest line
(with respect to the Borel subgroup $B\subset G_2$) $\ell\in\P(\C^7)$. The stabilizer
$\mathrm{Stab}_{G_2}(\ell)$ of $\ell$ is a parabolic subgroup $G_2\supset P\supset B$. Thus we obtain
a closed embedding of the parabolic flag variety $G_2/P$ into $\o{Q}$. Since $\dim G_2/P=5$ and
$\dim\P(\C^7)=5$, we see that $G_2/P\simeq\o{Q}$. Hence the number of $B$-orbits in $\o{Q}$ is the
same as the number of $B$-orbits in $G_2/P$, that is~6 (the cardinality of the quotient of the Weyl
group of $G_2$ modulo its parabolic subgroup $S_2$).

\begin{lem} \label{ls count 5}
	If $q \in Q$ lies in a $B$-orbit $\bO$, then for any $\la \ne 0$ the point $\la q$ also lies in $\bO$.
\end{lem}

\begin{proof}
  Set $\o{\bO} = \{l \in \P(Q) : l \cap \bO \ne \vn\}$. Then $\o{\bO}$ is a single $B$-orbit
  in $\o{Q}$. Clearly it is enough to show the lemma for the points in $l \cap \bO$ for some
  line $l \in \o{\bO}$. Fix a maximal torus $T \sb B$. Since $\o{\bO}$ is a $B$-orbit in the
  parabolic flag variety $G_2/P$, it contains a unique $T$-fixed point. It corresponds to a
  nonzero weight space in $\C^7$. Any weight space in $\C^7$ has dimension~1.
  Hence if we remove the origin from this weight space, what remains forms a single $T$-orbit.
\end{proof}

\begin{cor} \label{ls count 6}
	$\# (\text{$B$-orbits in $Q$}) = \#(\text{$B$-orbits in $\o{Q}$}) + 1 = 7$.
\end{cor}

\begin{proof}
  Since a $B$-orbit in $Q$ together with each point contains the whole line passing through
  this point (with the origin removed), the set of $B$-orbits in $Q$, except the zero-orbit, is in a
  natural bijection with the set of $B$-orbits in $\o{Q}$. And the number of latter orbits equals~$6$. 
\end{proof}

Corollary \ref{ls count 6} completes the proof of the second part of Theorem \ref{theo}.

\section{Appendix} \label{ap}


\subsection{Lie subalgebras vs closed subgroups} \label{ap sub}

Let $G$ be a complex linear algebraic group. It is well known that for any closed subgroup $H \sb G$ its tangent space $\h = \Lie H = T_\se H \sb T_\se G \sb \g$ is a Lie subalgebra in $\g$ (see \cite[\S4.4.8]{s}), where $\se$ is a neutral element in $G$. We begin with a simple lemma.

\begin{lem} \label{lie unique}
	Let $H_1, H_2 \sb G$ be two closed connected subgroups. Suppose $\Lie H_1 = \Lie H_2$. Then $H_1 = H_2$.
\end{lem}

\begin{proof}
  Since $G$ and $H_1, H_2$ are smooth algebraic varieties, they can be considered also as
  complex varieties and Lie groups. Then the lemma follows from one of the fundamental theorems
  of Lie theory, see \cite[Theorem 3.40]{k}.
\end{proof}

The existence result is more complicated. We prove it only in generality necessary for our paper.

\begin{prop} \label{lie exist}
  Let $G$ be a semisimple complex algebraic group. Let $\h \sb \g$ be a nilpotent Lie
  subalgebra normalized by a maximal torus $T \sb G$ such that $\h \cap \Lie T = 0$.
  Then there is a closed algebraic subgroup $H \sb G$ such that $\Lie H = \h$. 
\end{prop}

\begin{proof}
  Since $\h$ is normalized by $T$, there is a weight decomposition
  $\h = \bigoplus_{\al \in X^\ast(T)} \h_\al$, where $X^\ast(T)$ stands for the character lattice of
  $T$ (see \cite[Theorem 3.2.3]{s}). Denote by $R$ the set of roots of $G$, and set $R_\h = \{\al \in R : \h_\al \ne 0\}$. Then $\h = \h_0 \oplus \bigoplus_{\al \in R_\h} \g_\al$, where $\g_\al$ is 1-dimensional weight
  subspace in $\g$ of weight $\al$. Since $\h_0 \sb \g_0 = \Lie T$, we have $\h_0 = 0$, because
  we assume $\h \cap \Lie T = 0$.
	
	Set $D_0(\h) = \h$, $D_k(\h) = [\h, D_{k-1}(\h)]$ for $k \br 1$. Define $R_{D_k(\h)}$ similarly to $R_\h$. 
	
	\begin{lem} \label{sum of 2 roots}
		Suppose $\al \in R_\h, \be \in R_{D_k(\h)}$ and $i \al + j \be \in R$, where $i, j \in \Z_{> 0}$. Then $i \al + j \be \in R_{D_{k+1}}(\h)$.  
	\end{lem}
	
	\begin{proof}
	  It is enough to deal only with rank two root systems. It will be clear from the proof that it is enough to consider only the case $k = 0$. Clearly $\al \ne -\be$, otherwise $[\h_\al, \h_\be] \sb \h_0$ is non-zero and we get a contradiction. Then without loss of generality we may assume that $\al, \be$ are positive roots. Looking explicitly at all rank two reduced root systems, one can see that if $i \al + j \be \in R$, then $(i - 1) \al + j \be, i \al + (j - 1)\be \in R$. Thus we may suppose that $i = j = 1$. By \cite[Theorem 6.44~(7)]{k} $[\g_\al, \g_\be] = \g_{\al +\be}$ in this case, so $\h_{\al + \be} = \g_{\al +\be} = [\h_\al, \h_\be] \ne 0$ and hence $\al + \be \in R_{D_1(\h)}$.
	\end{proof}
	
	Let $R_\h = \{\al_1, \ld, \al_h\}$. For any $\al \in R$ denote by $U_\al$ the corresponding
        unipotent subgroup in $G$ and fix an isomorphism $u_\al \colon \bG_a \iso U_\al$ such that for
        any $x \in \C$, $t \in T$ we have $t u_\al(x) t^{-1} = u_\al(\al(t) x)$ (cf.~\cite[Prop.~8.2.1]{s}).
        Then there is a morphism $\phi \colon \bG_a^h \to G$, $(x_1, \ld, x_h) \mapsto u_{\al_1}(x_1) \times \ld \times u_{\al_h}(x_h)$. 
	
	\begin{lem}
		The image of $\phi$ is a closed subgroup in $G$.
	\end{lem}
	
	\begin{proof}
	  Due to \cite[Prop.~2.2.5]{s} it is enough to prove that $\Im(\phi)$ is a subgroup in
          $G$. Equivalently, it is enough to prove that for any $y \in \C$ and $1 \mr j \mr h$
          the element $u_{\al_j}(y) \times u_{\al_1}(x_1) \times \ld \times u_{\al_h}(x_h)$ lies in $\Im(\phi)$.
          Since for any $\al, \be \in R_\h$ we have $\al \ne -\be$, there is a system of positive
          roots in $R$ containing $\al$ and $\be$. Thus one can apply \cite[Prop.~8.2.3]{s},
          i.e.~we can permute factors, so at each step there would appear the product of factors
          with "greater" index. More precisely, let us put $u_{\al_i}(y)$ at the $i^{th}$ position.
          While moving past $1^{st}, 2^{nd}, \ld, (i - 1)^{th}$ factors there could appear some
          other factors of the form $u_{k_i \al_i + k_l \al_l}(x_{i, l})$ $(k_i, k_l > 0)$
          such that by Lemma \ref{sum of 2 roots} $k_i \al_i + k_l \al_l \in R_{D_1(\h)}$. Then we move
          each factor $u_{k_i \al_i + k_l \al_l}(x_{i, l})$ that appeared above back to its position,
          i.e.\ move it past other factors to the initial factor $u_{k_i \al_i + k_l \al_l}$.
          While doing this, there could appear some other factors of the form
          $u_{k_{i, l}(k_i \al_i + k_l \al_l) + k_m \al_m}(x_{i, l, m})$ $(k_{i, l}, k_m > 0)$ such that
          by Lemma \ref{sum of 2 roots} $k_{i, l}(k_i \al_i + k_l \al_l) + k_m \al_m \in R_{D_2(\h)}$.
          And so on. Since $\h$ is nilpotent, $D_N(\h) = 0$ for some $N \br 1$. Therefore,
          the factors appearing at the $N^{th}$ step, are trivial, because their weights lie in
          $R_{D_N(\h)} = \vn$. So finally we can permute all the factors and get an element of $\Im(\phi)$.
	\end{proof}
	
	Now let us show that $H = \Im(\phi)$ is the desired subgroup. Since
        $\g_\al = \Lie \Im(u_{\al})$ for any $\al \in R$, by construction of $\phi$ we have
        $\h \sb \Lie H$. On the other hand, $\dim H \mr \dim \bG_a^h = h = \dim \h$. So $\h = \Lie H$ and we are done.
\end{proof}

\begin{rem}
	Prop.~\ref{lie exist} is a generalization of \cite[Prop.~8.2.4]{s}, and the proof is based on the technique similar to \cite[Prop.~8.2.3]{s}. 
\end{rem}


\subsection{Some facts about $G_2$} \label{ap g2}

Consider a vector space $\g_2:= V \oplus \sl_3(\C) \oplus V^t$, where $V = \C^3$. One can
view $V$ as $3 \times 1$-matrices (i.e.\ columns) with entries in $\C$, and $V^t$ as $1 \times 3$-matrices (i.e.\ rows) with entries in $\C$. Introduce a Lie bracket on $\g_2$ as follows: let $x \in \sl_3(\C)$, $v, w \in V$, then
\[[x, v] = xv, \: [x, v^t] = - v^t x, \: [v, w] = 2(v \times w)^t, \: [v^t, w^t] = 2 v \times w, \: [v, w^t] = -3vw^t + w^tv \cd \Id\]
It is well known (cf.\ a similar construction in \cite[\S 22]{fh}), that $\g_2$
is a well-defined simple Lie algebra. Its Cartan subalgebra is the diagonal Cartan subalgebra $\h \sb \sl_3(\C)$. By $E_{ij} \in \gl_3(\C)$ I denote the elementary matrix with $1$ in $i^{th}$-row and $j^{th}$ column ($i\ne j$) and zero in other cells.  Denote by $e_1, e_2, e_3$ the standard base in $V = \C^3$ and $f_i := e_i^t$.

Elements of $\h$ can be presented in the form $a_1 E_{11} + a_2 E_{22} + a_3 E_{33}$, where $a_i \in \C$ and $a_1 + a_2 + a_3 = 0$. So $a_i$ can be viewed as elements of $\h^\ast$. It is easy to realize that we have the following root decomposition:
\begin{itemize}
	\item $e_i$ has weight $a_i$;
	\item $f_i$ has weight $-a_i$;
	\item $E_{ij}$ has weight $a_i - a_j$, where $i \ne j$.
\end{itemize}
One can write the Killing form on $\h$ (using an explicit computation):
\[(a_1 E_{11} + a_2 E_{22} + a_3 E_{33}, b_1 E_{11} + b_2 E_{22} + b_3 E_{33}) = 6 \sum \l_{i = 1}^3 a_ib_i - 2 \sum \l_{i \ne j} a_i b_j\]
By this Killing form $\pm a_i$ corresponds to $\pm \f{3E_{ii} - (E_{11} + E_{22} + E_{33})}{24}$, $a_i - a_j$ corresponds to $\f{E_{ii} - E_{jj}}{8}$. Thus, 
\[\|a_i\|^2 = \f{1}{12}, \quad \|a_i - a_j\|^2 = \f{1}{4} = 3 \cd \|a_i\|^2\]
Thus $a_i$ is shorter than $a_k - a_l$. Actually the Killing form restricted to the plane $\R^2$ is the standard Euclidean form such that the length of $a_i$ is $\f{1}{12}$.

\begin{figure}[h]
\[\begin{tikzpicture}[x=0.75pt,y=0.75pt,yscale=-1,xscale=1]
	
	\draw  [color={rgb, 255:red, 0; green, 0; blue, 0 }  ,draw opacity=1 ][dash pattern={on 0.84pt off 2.51pt}] (339.4,200.12) -- (270.28,240.25) -- (200.96,200.45) -- (200.76,120.53) -- (269.88,80.4) -- (339.2,120.19) -- cycle ;
	\draw [color={rgb, 255:red, 0; green, 0; blue, 0 }  ,draw opacity=1 ]   (270.08,160.32) -- (337.66,199.12) ;
	\draw [shift={(339.4,200.12)}, rotate = 209.86] [color={rgb, 255:red, 0; green, 0; blue, 0 }  ,draw opacity=1 ][line width=0.75]    (10.93,-3.29) .. controls (6.95,-1.4) and (3.31,-0.3) .. (0,0) .. controls (3.31,0.3) and (6.95,1.4) .. (10.93,3.29)   ;
	\draw [color={rgb, 255:red, 0; green, 0; blue, 0 }  ,draw opacity=1 ]   (270.08,160.32) -- (337.47,121.19) ;
	\draw [shift={(339.2,120.19)}, rotate = 149.86] [color={rgb, 255:red, 0; green, 0; blue, 0 }  ,draw opacity=1 ][line width=0.75]    (10.93,-3.29) .. controls (6.95,-1.4) and (3.31,-0.3) .. (0,0) .. controls (3.31,0.3) and (6.95,1.4) .. (10.93,3.29)   ;
	\draw [color={rgb, 255:red, 0; green, 0; blue, 0 }  ,draw opacity=1 ]   (270.08,160.32) -- (269.89,82.4) ;
	\draw [shift={(269.88,80.4)}, rotate = 89.86] [color={rgb, 255:red, 0; green, 0; blue, 0 }  ,draw opacity=1 ][line width=0.75]    (10.93,-3.29) .. controls (6.95,-1.4) and (3.31,-0.3) .. (0,0) .. controls (3.31,0.3) and (6.95,1.4) .. (10.93,3.29)   ;
	\draw [color={rgb, 255:red, 0; green, 0; blue, 0 }  ,draw opacity=1 ]   (270.08,160.32) -- (270.27,238.25) ;
	\draw [shift={(270.28,240.25)}, rotate = 269.86] [color={rgb, 255:red, 0; green, 0; blue, 0 }  ,draw opacity=1 ][line width=0.75]    (10.93,-3.29) .. controls (6.95,-1.4) and (3.31,-0.3) .. (0,0) .. controls (3.31,0.3) and (6.95,1.4) .. (10.93,3.29)   ;
	\draw [color={rgb, 255:red, 0; green, 0; blue, 0 }  ,draw opacity=1 ]   (270.08,160.32) -- (202.69,199.45) ;
	\draw [shift={(200.96,200.45)}, rotate = 329.86] [color={rgb, 255:red, 0; green, 0; blue, 0 }  ,draw opacity=1 ][line width=0.75]    (10.93,-3.29) .. controls (6.95,-1.4) and (3.31,-0.3) .. (0,0) .. controls (3.31,0.3) and (6.95,1.4) .. (10.93,3.29)   ;
	\draw [color={rgb, 255:red, 0; green, 0; blue, 0 }  ,draw opacity=1 ]   (202.5,121.52) -- (270.08,160.32) ;
	\draw [shift={(200.76,120.53)}, rotate = 29.86] [color={rgb, 255:red, 0; green, 0; blue, 0 }  ,draw opacity=1 ][line width=0.75]    (10.93,-3.29) .. controls (6.95,-1.4) and (3.31,-0.3) .. (0,0) .. controls (3.31,0.3) and (6.95,1.4) .. (10.93,3.29)   ;
	\draw [color={rgb, 255:red, 0; green, 0; blue, 0 }  ,draw opacity=1 ] [dash pattern={on 0.84pt off 2.51pt}]  (200.76,120.53) -- (339.2,120.19) ;
	\draw [color={rgb, 255:red, 0; green, 0; blue, 0 }  ,draw opacity=1 ] [dash pattern={on 0.84pt off 2.51pt}]  (200.96,200.45) -- (339.4,200.12) ;
	\draw [color={rgb, 255:red, 0; green, 0; blue, 0 }  ,draw opacity=1 ] [dash pattern={on 0.84pt off 2.51pt}]  (269.88,80.4) -- (339.4,200.12) ;
	\draw [color={rgb, 255:red, 0; green, 0; blue, 0 }  ,draw opacity=1 ] [dash pattern={on 0.84pt off 2.51pt}]  (200.96,200.45) -- (269.88,80.4) ;
	\draw [color={rgb, 255:red, 0; green, 0; blue, 0 }  ,draw opacity=1 ] [dash pattern={on 0.84pt off 2.51pt}]  (200.76,120.53) -- (270.28,240.25) ;
	\draw [color={rgb, 255:red, 0; green, 0; blue, 0 }  ,draw opacity=1 ] [dash pattern={on 0.84pt off 2.51pt}]  (339.2,120.19) -- (270.28,240.25) ;
	\draw [color={rgb, 255:red, 0; green, 0; blue, 0 }  ,draw opacity=1 ]   (247.89,122.29) -- (270.08,160.32) ;
	\draw [shift={(246.88,120.56)}, rotate = 59.74] [color={rgb, 255:red, 0; green, 0; blue, 0 }  ,draw opacity=1 ][line width=0.75]    (10.93,-3.29) .. controls (6.95,-1.4) and (3.31,-0.3) .. (0,0) .. controls (3.31,0.3) and (6.95,1.4) .. (10.93,3.29)   ;
	\draw [color={rgb, 255:red, 0; green, 0; blue, 0 }  ,draw opacity=1 ]   (270.08,160.32) -- (292.27,122.29) ;
	\draw [shift={(293.28,120.56)}, rotate = 120.26] [color={rgb, 255:red, 0; green, 0; blue, 0 }  ,draw opacity=1 ][line width=0.75]    (10.93,-3.29) .. controls (6.95,-1.4) and (3.31,-0.3) .. (0,0) .. controls (3.31,0.3) and (6.95,1.4) .. (10.93,3.29)   ;
	\draw [color={rgb, 255:red, 0; green, 0; blue, 0 }  ,draw opacity=1 ]   (270.08,160.32) -- (314.08,160.17) ;
	\draw [shift={(316.08,160.16)}, rotate = 179.8] [color={rgb, 255:red, 0; green, 0; blue, 0 }  ,draw opacity=1 ][line width=0.75]    (10.93,-3.29) .. controls (6.95,-1.4) and (3.31,-0.3) .. (0,0) .. controls (3.31,0.3) and (6.95,1.4) .. (10.93,3.29)   ;
	\draw [color={rgb, 255:red, 0; green, 0; blue, 0 }  ,draw opacity=1 ]   (270.08,160.32) -- (225.28,160.55) ;
	\draw [shift={(223.28,160.56)}, rotate = 359.71] [color={rgb, 255:red, 0; green, 0; blue, 0 }  ,draw opacity=1 ][line width=0.75]    (10.93,-3.29) .. controls (6.95,-1.4) and (3.31,-0.3) .. (0,0) .. controls (3.31,0.3) and (6.95,1.4) .. (10.93,3.29)   ;
	\draw [color={rgb, 255:red, 0; green, 0; blue, 0 }  ,draw opacity=1 ]   (270.08,160.32) -- (292.28,198.83) ;
	\draw [shift={(293.28,200.56)}, rotate = 240.03] [color={rgb, 255:red, 0; green, 0; blue, 0 }  ,draw opacity=1 ][line width=0.75]    (10.93,-3.29) .. controls (6.95,-1.4) and (3.31,-0.3) .. (0,0) .. controls (3.31,0.3) and (6.95,1.4) .. (10.93,3.29)   ;
	\draw [color={rgb, 255:red, 0; green, 0; blue, 0 }  ,draw opacity=1 ]   (270.08,160.32) -- (247.89,198.04) ;
	\draw [shift={(246.88,199.76)}, rotate = 300.47] [color={rgb, 255:red, 0; green, 0; blue, 0 }  ,draw opacity=1 ][line width=0.75]    (10.93,-3.29) .. controls (6.95,-1.4) and (3.31,-0.3) .. (0,0) .. controls (3.31,0.3) and (6.95,1.4) .. (10.93,3.29)   ;
	\draw [color={rgb, 255:red, 0; green, 0; blue, 0 }  ,draw opacity=1 ] [dash pattern={on 4.5pt off 4.5pt}]  (194.08,96.4) -- (350.08,227.96) ;
	
	\draw (258.08,66.4) node [anchor=north west][inner sep=0.75pt]  [font=\tiny]  {$a_{2} -a_{3}$};
	\draw (292.08,110) node [anchor=north west][inner sep=0.75pt]  [font=\tiny]  {$-a_{3}$};
	\draw (340.48,111) node [anchor=north west][inner sep=0.75pt]  [font=\tiny]  {$a_{1} -a_{3}$};
	\draw (319.28,156) node [anchor=north west][inner sep=0.75pt]  [font=\tiny]  {$a_{1}$};
	\draw (342.48,193.4) node [anchor=north west][inner sep=0.75pt]  [font=\tiny]  {$a_{1} -a_{2}$};
	\draw (295.28,203.96) node [anchor=north west][inner sep=0.75pt]  [font=\tiny]  {$-a_{2}$};
	\draw (258.48,246.6) node [anchor=north west][inner sep=0.75pt]  [font=\tiny]  {$a_{3} -a_{2}$};
	\draw (236.88,204.2) node [anchor=north west][inner sep=0.75pt]  [font=\tiny]  {$a_{3}$};
	\draw (169.68,201.4) node [anchor=north west][inner sep=0.75pt]  [font=\tiny]  {$a_{3} -a_{1}$};
	\draw (202.88,157.8) node [anchor=north west][inner sep=0.75pt]  [font=\tiny]  {$-a_{1}$};
	\draw (167.68,111.8) node [anchor=north west][inner sep=0.75pt]  [font=\tiny]  {$a_{2} -a_{1}$};
	\draw (237.28,110.8) node [anchor=north west][inner sep=0.75pt]  [font=\tiny]  {$a_{2}$};
	
\end{tikzpicture}\]
\caption{Roots of $\g_2$}
\label{3}
\end{figure}


\subsection{Seven-dimensional representation of $\g_2$} \label{ap 7rep}

Consider a polarization given by the dashed line in Figure \ref{3}, so the system of positive roots consist of all roots above the dashed line. Then the Weyl chamber $C_+$ between roots $-a_3$ and $a_1 - a_3$ in Figure \ref{3} is the
fundamental chamber. Consider the root $(-a_3)$ in this chamber. Then there is a unique irreducible
representation $V_{-a_3}$ with highest weight $(-a_3)$. Its weights are depicted in Figure \ref{4}:

\begin{figure}[h]
 \[\begin{tikzpicture}[x=0.75pt,y=0.75pt,yscale=-1,xscale=1]
	
	\draw   (271.04,100.24) -- (240.56,153.03) -- (179.6,153.03) -- (149.12,100.24) -- (179.6,47.45) -- (240.56,47.45) -- cycle ;
	
	\draw (214.72,92) node [anchor=north west][inner sep=0.75pt]  [font=\tiny]  {$0$};
	\draw (203.92,97.4) node [anchor=north west][inner sep=0.75pt]    {$\bullet $};
	\draw (240.32,36.8) node [anchor=north west][inner sep=0.75pt]  [font=\tiny]  {$-a_{3}$};
	\draw (168.72,37.2) node [anchor=north west][inner sep=0.75pt]  [font=\tiny]  {$a_{2}$};
	\draw (127.92,92) node [anchor=north west][inner sep=0.75pt]  [font=\tiny]  {$-a_{1}$};
	\draw (169.12,155) node [anchor=north west][inner sep=0.75pt]  [font=\tiny]  {$a_{3}$};
	\draw (232.72,157) node [anchor=north west][inner sep=0.75pt]  [font=\tiny]  {$-a_{2}$};
	\draw (274.32,92.2) node [anchor=north west][inner sep=0.75pt]  [font=\tiny]  {$a_{1}$};

\end{tikzpicture}\]
\caption{Weights of $V_{-a_3}$}
\label{4}
\end{figure}
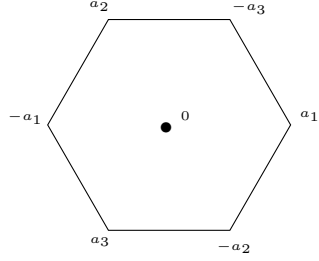

As follows from the representation theory of $\sl_2$, the weights on the perimeter have
multiplicity~$1$.

\begin{lem} \label{ap 7rep 1}
	The zero-weight space in $V_{-a_3}$ is one-dimensional.
\end{lem}

\begin{proof}
  Denote by $v$ the highest weight vector. Then we need to show that
  $f_1f_2 \cd v \sim f_2f_1 \cd v$. First note that $E_{12} f_2 \cd v \sim f_1$, because the weight space corresponding to $a_1$ is one-dimensional. Set $w = f_1 \cd v$. Thus, it is enough to show that $f_1 E_{12} \cd w \sim f_2 \cd w$. Indeed, $[E_{12}, f_1] = -f_2$, so $f_1 E_{12} \cd w = f_2 \cd w - E_{12} f_1 \cd w$. Since $f_1 \cd w = 0$, we have $f_1 E_{12} \cd w = f_2 \cd w$. 
\end{proof}

Hence we get $\dim V_{-a_3} = 7$. The next step is to show that there is a non-degenerate $q \in S^2 V_{-a_3}$ invariant with respect to $\g_2$. We will find this $q$ explicitly.

\begin{lem} \label{ap 7rep 2}
	One can choose weight vectors $v, \t{v}, u, \ld$ such that the diagram in Figure \ref{5}, where $E_{32}$, $-E_{12}$, $-E_{32}$, $E_{12}$ act by the dotted arrows, commutes.

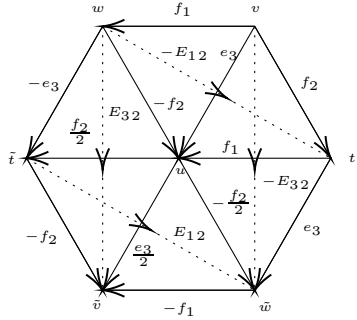
\begin{figure}[h]
        \[\begin{tikzpicture}[x=0.75pt,y=0.75pt,yscale=-1,xscale=1]
        	
        	\draw  [color={rgb, 255:red, 0; green, 0; blue, 0 }  ,draw opacity=1 ] (326.64,140.8) -- (288.48,206.9) -- (212.16,206.9) -- (174,140.8) -- (212.16,74.7) -- (288.48,74.7) -- cycle ;
        	\draw [color={rgb, 255:red, 0; green, 0; blue, 0 }  ,draw opacity=1 ]   (288.48,74.7) -- (214.16,74.7) ;
        	\draw [shift={(212.16,74.7)}, rotate = 360] [color={rgb, 255:red, 0; green, 0; blue, 0 }  ,draw opacity=1 ][line width=0.75]    (10.93,-3.29) .. controls (6.95,-1.4) and (3.31,-0.3) .. (0,0) .. controls (3.31,0.3) and (6.95,1.4) .. (10.93,3.29)   ;
        	\draw [color={rgb, 255:red, 0; green, 0; blue, 0 }  ,draw opacity=1 ]   (212.16,74.7) -- (175,139.07) ;
        	\draw [shift={(174,140.8)}, rotate = 300] [color={rgb, 255:red, 0; green, 0; blue, 0 }  ,draw opacity=1 ][line width=0.75]    (10.93,-3.29) .. controls (6.95,-1.4) and (3.31,-0.3) .. (0,0) .. controls (3.31,0.3) and (6.95,1.4) .. (10.93,3.29)   ;
        	\draw [color={rgb, 255:red, 0; green, 0; blue, 0 }  ,draw opacity=1 ]   (174,140.8) -- (211.16,205.16) ;
        	\draw [shift={(212.16,206.9)}, rotate = 240] [color={rgb, 255:red, 0; green, 0; blue, 0 }  ,draw opacity=1 ][line width=0.75]    (10.93,-3.29) .. controls (6.95,-1.4) and (3.31,-0.3) .. (0,0) .. controls (3.31,0.3) and (6.95,1.4) .. (10.93,3.29)   ;
        	\draw [color={rgb, 255:red, 0; green, 0; blue, 0 }  ,draw opacity=1 ]   (288.48,206.9) -- (214.16,206.9) ;
        	\draw [shift={(212.16,206.9)}, rotate = 360] [color={rgb, 255:red, 0; green, 0; blue, 0 }  ,draw opacity=1 ][line width=0.75]    (10.93,-3.29) .. controls (6.95,-1.4) and (3.31,-0.3) .. (0,0) .. controls (3.31,0.3) and (6.95,1.4) .. (10.93,3.29)   ;
        	\draw [color={rgb, 255:red, 0; green, 0; blue, 0 }  ,draw opacity=1 ]   (288.48,74.7) -- (325.64,139.07) ;
        	\draw [shift={(326.64,140.8)}, rotate = 240] [color={rgb, 255:red, 0; green, 0; blue, 0 }  ,draw opacity=1 ][line width=0.75]    (10.93,-3.29) .. controls (6.95,-1.4) and (3.31,-0.3) .. (0,0) .. controls (3.31,0.3) and (6.95,1.4) .. (10.93,3.29)   ;
        	\draw [color={rgb, 255:red, 0; green, 0; blue, 0 }  ,draw opacity=1 ]   (326.64,140.8) -- (289.48,205.16) ;
        	\draw [shift={(288.48,206.9)}, rotate = 300] [color={rgb, 255:red, 0; green, 0; blue, 0 }  ,draw opacity=1 ][line width=0.75]    (10.93,-3.29) .. controls (6.95,-1.4) and (3.31,-0.3) .. (0,0) .. controls (3.31,0.3) and (6.95,1.4) .. (10.93,3.29)   ;
        	\draw [color={rgb, 255:red, 0; green, 0; blue, 0 }  ,draw opacity=1 ]   (212.16,74.7) -- (249.32,139.07) ;
        	\draw [shift={(250.32,140.8)}, rotate = 240] [color={rgb, 255:red, 0; green, 0; blue, 0 }  ,draw opacity=1 ][line width=0.75]    (10.93,-3.29) .. controls (6.95,-1.4) and (3.31,-0.3) .. (0,0) .. controls (3.31,0.3) and (6.95,1.4) .. (10.93,3.29)   ;
        	\draw [color={rgb, 255:red, 0; green, 0; blue, 0 }  ,draw opacity=1 ]   (288.48,74.7) -- (251.32,139.07) ;
        	\draw [shift={(250.32,140.8)}, rotate = 300] [color={rgb, 255:red, 0; green, 0; blue, 0 }  ,draw opacity=1 ][line width=0.75]    (10.93,-3.29) .. controls (6.95,-1.4) and (3.31,-0.3) .. (0,0) .. controls (3.31,0.3) and (6.95,1.4) .. (10.93,3.29)   ;
        	\draw [color={rgb, 255:red, 0; green, 0; blue, 0 }  ,draw opacity=1 ]   (176,140.8) -- (250.32,140.8) ;
        	\draw [shift={(174,140.8)}, rotate = 0] [color={rgb, 255:red, 0; green, 0; blue, 0 }  ,draw opacity=1 ][line width=0.75]    (10.93,-3.29) .. controls (6.95,-1.4) and (3.31,-0.3) .. (0,0) .. controls (3.31,0.3) and (6.95,1.4) .. (10.93,3.29)   ;
        	\draw [color={rgb, 255:red, 0; green, 0; blue, 0 }  ,draw opacity=1 ]   (326.64,140.8) -- (252.32,140.8) ;
        	\draw [shift={(250.32,140.8)}, rotate = 360] [color={rgb, 255:red, 0; green, 0; blue, 0 }  ,draw opacity=1 ][line width=0.75]    (10.93,-3.29) .. controls (6.95,-1.4) and (3.31,-0.3) .. (0,0) .. controls (3.31,0.3) and (6.95,1.4) .. (10.93,3.29)   ;
        	\draw [color={rgb, 255:red, 0; green, 0; blue, 0 }  ,draw opacity=1 ]   (250.32,140.8) -- (287.48,205.16) ;
        	\draw [shift={(288.48,206.9)}, rotate = 240] [color={rgb, 255:red, 0; green, 0; blue, 0 }  ,draw opacity=1 ][line width=0.75]    (10.93,-3.29) .. controls (6.95,-1.4) and (3.31,-0.3) .. (0,0) .. controls (3.31,0.3) and (6.95,1.4) .. (10.93,3.29)   ;
        	\draw [color={rgb, 255:red, 0; green, 0; blue, 0 }  ,draw opacity=1 ]   (250.32,140.8) -- (213.16,205.16) ;
        	\draw [shift={(212.16,206.9)}, rotate = 300] [color={rgb, 255:red, 0; green, 0; blue, 0 }  ,draw opacity=1 ][line width=0.75]    (10.93,-3.29) .. controls (6.95,-1.4) and (3.31,-0.3) .. (0,0) .. controls (3.31,0.3) and (6.95,1.4) .. (10.93,3.29)   ;
        	\draw [color={rgb, 255:red, 0; green, 0; blue, 0 }  ,draw opacity=1 ] [dash pattern={on 0.84pt off 2.51pt}]  (212.16,74.7) -- (212.16,206.9) ;
        	\draw [shift={(212.16,146.8)}, rotate = 270] [color={rgb, 255:red, 0; green, 0; blue, 0 }  ,draw opacity=1 ][line width=0.75]    (10.93,-3.29) .. controls (6.95,-1.4) and (3.31,-0.3) .. (0,0) .. controls (3.31,0.3) and (6.95,1.4) .. (10.93,3.29)   ;
        	\draw [color={rgb, 255:red, 0; green, 0; blue, 0 }  ,draw opacity=1 ] [dash pattern={on 0.84pt off 2.51pt}]  (212.16,74.7) -- (326.64,140.8) ;
        	\draw [shift={(274.6,110.75)}, rotate = 210] [color={rgb, 255:red, 0; green, 0; blue, 0 }  ,draw opacity=1 ][line width=0.75]    (10.93,-3.29) .. controls (6.95,-1.4) and (3.31,-0.3) .. (0,0) .. controls (3.31,0.3) and (6.95,1.4) .. (10.93,3.29)   ;
        	\draw  [dash pattern={on 0.84pt off 2.51pt}]  (174,140.8) -- (288.48,206.9) ;
        	\draw [shift={(236.44,176.85)}, rotate = 210] [color={rgb, 255:red, 0; green, 0; blue, 0 }  ][line width=0.75]    (10.93,-3.29) .. controls (6.95,-1.4) and (3.31,-0.3) .. (0,0) .. controls (3.31,0.3) and (6.95,1.4) .. (10.93,3.29)   ;
        	\draw  [dash pattern={on 0.84pt off 2.51pt}]  (288.48,74.7) -- (288.48,206.9) ;
        	\draw [shift={(288.48,146.8)}, rotate = 270] [color={rgb, 255:red, 0; green, 0; blue, 0 }  ][line width=0.75]    (10.93,-3.29) .. controls (6.95,-1.4) and (3.31,-0.3) .. (0,0) .. controls (3.31,0.3) and (6.95,1.4) .. (10.93,3.29)   ;
        	
        	\draw (247.12,144.4) node [anchor=north west][inner sep=0.75pt]  [font=\tiny]  {$u$};
        	\draw (285.12,63.2) node [anchor=north west][inner sep=0.75pt]  [font=\tiny]  {$v$};
        	\draw (205.12,63.6) node [anchor=north west][inner sep=0.75pt]  [font=\tiny]  {$w$};
        	\draw (205.52,210.6) node [anchor=north west][inner sep=0.75pt]  [font=\tiny]  {$\tilde{v}$};
        	\draw (289.12,211.4) node [anchor=north west][inner sep=0.75pt]  [font=\tiny]  {$\tilde{w}$};
        	\draw (334.32,137.4) node [anchor=north west][inner sep=0.75pt]  [font=\tiny]  {$t$};
        	\draw (163.36,136.6) node [anchor=north west][inner sep=0.75pt]  [font=\tiny]  {$\tilde{t}$};
        	\draw (246.26,61.8) node [anchor=north west][inner sep=0.75pt]  [font=\tiny,color={rgb, 255:red, 0; green, 0; blue, 0 }  ,opacity=1 ]  {$f_{1}$};
        	\draw (172.96,99.6) node [anchor=north west][inner sep=0.75pt]  [font=\tiny,color={rgb, 255:red, 0; green, 0; blue, 0 }  ,opacity=1 ]  {$-e_{3}$};
        	\draw (171.76,175.9) node [anchor=north west][inner sep=0.75pt]  [font=\tiny,color={rgb, 255:red, 0; green, 0; blue, 0 }  ,opacity=1 ]  {$-f_{2}$};
        	\draw (240.96,211.8) node [anchor=north west][inner sep=0.75pt]  [font=\tiny,color={rgb, 255:red, 0; green, 0; blue, 0 }  ,opacity=1 ]  {$-f_{1}$};
        	\draw (309.96,98.9) node [anchor=north west][inner sep=0.75pt]  [font=\tiny,color={rgb, 255:red, 0; green, 0; blue, 0 }  ,opacity=1 ]  {$f_{2}$};
        	\draw (311.36,173.3) node [anchor=north west][inner sep=0.75pt]  [font=\tiny,color={rgb, 255:red, 0; green, 0; blue, 0 }  ,opacity=1 ]  {$e_{3}$};
        	\draw (235.76,109) node [anchor=north west][inner sep=0.75pt]  [font=\tiny,color={rgb, 255:red, 0; green, 0; blue, 0 }  ,opacity=1 ]  {$-f_{2}$};
        	\draw (270.36,131) node [anchor=north west][inner sep=0.75pt]  [font=\tiny,color={rgb, 255:red, 0; green, 0; blue, 0 }  ,opacity=1 ]  {$f_{1}$};
        	\draw (192.96,120) node [anchor=north west][inner sep=0.75pt]  [font=\tiny,color={rgb, 255:red, 0; green, 0; blue, 0 }  ,opacity=1 ]  {$\frac{f_{2}}{2}$};
        	\draw (224.94,182.95) node [anchor=north west][inner sep=0.75pt]  [font=\tiny,color={rgb, 255:red, 0; green, 0; blue, 0 }  ,opacity=1 ]  {$\frac{e_{3}}{2}$};
        	\draw (269.56,85.4) node [anchor=north west][inner sep=0.75pt]  [font=\tiny,color={rgb, 255:red, 0; green, 0; blue, 0 }  ,opacity=1 ]  {$e_{3}$};
        	\draw (265.36,155.1) node [anchor=north west][inner sep=0.75pt]  [font=\tiny,color={rgb, 255:red, 0; green, 0; blue, 0 }  ,opacity=1 ]  {$-\frac{f_{2}}{2}$};
        	\draw (213.26,113.7) node [anchor=north west][inner sep=0.75pt]  [font=\tiny,color={rgb, 255:red, 0; green, 0; blue, 0 }  ,opacity=1 ]  {$E_{3}{}_{2}$};
        	\draw (290.76,148.6) node [anchor=north west][inner sep=0.75pt]  [font=\tiny,color={rgb, 255:red, 0; green, 0; blue, 0 }  ,opacity=1 ]  {$-E_{3}{}_{2}$};
        	\draw (245.96,173.9) node [anchor=north west][inner sep=0.75pt]  [font=\tiny,color={rgb, 255:red, 0; green, 0; blue, 0 }  ,opacity=1 ]  {$E_{1}{}_{2}$};
        	\draw (239.56,84) node [anchor=north west][inner sep=0.75pt]  [font=\tiny,color={rgb, 255:red, 0; green, 0; blue, 0 }  ,opacity=1 ]  {$-E_{1}{}_{2}$};
        	        	
        \end{tikzpicture}\]
        \caption{Action of certain elements of $\g_2$ on $V_{-a_3}$}
        \label{5}
        \end{figure}
\end{lem}

\begin{proof}[Plan of the proof]
	The proof is quite technical and uses three facts: $[E_{12}, f_1] = -f_2$, $[f_1, f_2] = 2e_3$ and $[e_3, f_2] = -3E_{32}$. One should prove the commutativity of triangles (literally as in Lemma \ref{ap 7rep 1} we have proved $f_1 E_{12} \cd w = f_2 \cd w$).
\end{proof}

Now we are ready to give an explicit formula for $q$.

\begin{lem} \label{ap 7rep 3}
  The quadratic form $q = u^2 - 4 v \t{v} - 4 w \t{w}  - 4 t \t{t} \in S^2 V_{-a_3}$ is
  $\g_2$-invariant.
\end{lem}

\begin{proof}
	Let us prove first that $q$ is $f_1$-, $f_2$- and $f_3$-invariant. For $f_1$ we have $f_1 \cd v = w$, $f_1 \cd u = 2 \t{t}$, $f_1 \cd t = u$, $f_1 \cd \t{w} = -\t{v}$ and zero for other vectors. Clearly, this implies that $f_1 \cd q = 0$. Similarly one can show $f_2 \cd q = 0$. 
	
	Now let us consider $f_3$. We need to compute $f_3 \cd \t{t}$. Write $[e_3, f_3] = -3 E_{33} + \Id = E_{11} + E_{22} - 2E_{33}$. So $e_3f_3 \cd \t{t} = [e_3, f_3] \cd \t{t}$ (because $e_3 \cd \t{t} = 0$). Thus $e_3 f_3 \cd \t{t} = - \t{t}$, because $\t{t}$ has weight $-a_1$. Since $e_3 \cd w = - \t{t}$, we have $f_3 \cd \t{t} = w$. Similarly one can show $f_3 \cd \t{w} = -t$, $f_3 \cd \t{v} =  -u$, $f_3 \cd u = - 2v$. This implies $f_3 \cd q = 0$.
	
	Since $q$ is invariant under the action of $f_1, f_2, f_3$, and $f_1, f_2, f_3$ generate
        $\g_2$ (indeed, they generate $V$, and $V$ together with $V^t$ generate $\sl_3(\C)$,
        $q$ is $\g_2$-invariant.
\end{proof}

Clearly $q$ is non-degenerate. Therefore, we may consider the $\g_2$-action on $V_{-a_3}^t \simeq V_{-a_3}$ preserving
a non-degenerate symmetric $2$-form. This form is unique up to a scalar, because the zero-weight subspace in $S^2 V_{-a_3}$ is one-dimensional.



\newcommand{\Addresses}{{
		\bigskip
		
		\textsc{National Research University Higher School of Economics, Russian Federation, Department of
			Mathematics, 6 Usacheva st, 119048 Moscow}\par\nopagebreak
		\textit{E-mail address}, N.~Kononenko: \: \texttt{nakononenko@edu.hse.ru}
		
}}

\Addresses

\end{document}